\documentclass[a4paper]{ifacconf}

\usepackage{amsmath, amssymb, amsfonts}
\usepackage{bbm}
\allowdisplaybreaks
\usepackage{graphicx}
\usepackage{subcaption}
\usepackage{tikz}
\usetikzlibrary{arrows.meta}
\usepackage{xcolor}
\usepackage{natbib}

\definecolor{dgreen}{rgb}{0.,0.8,0.}
\newcommand{\green}[1]{{\color{black}#1}}

\newcommand{\qedsymbol}{$\blacksquare$}

\newtheorem{theorem}{Theorem}

\newtheorem{proposition}[theorem]{Proposition}

\newtheorem{assumption}[theorem]{Assumption}
\newtheorem{remark}[theorem]{Remark}

\newenvironment{proof}[1][Proof]{%
  \par\noindent\textbf{#1. }\rmfamily
}{\hfill\qedsymbol\par}
\newcounter{example}
\renewcommand{\theexample}{\arabic{example}}

\newcommand*{\dif}{\mathop{}\!\mathrm{d}}
\newcommand{\reals}{\mathbb{R}}
\newcommand{\naturals}{\mathbb{N}}

\newcommand{\diag}{\mathrm{diag}}

\newcommand{\argmin}{\mathrm{argmin}}
\newcommand{\Sparse}{\mathcal{S}}
\newcommand{\Tr}{\mathrm{tr}}
\newcommand{\Sym}{\mathrm{Sym}}
\newcommand{\col}{\mathrm{col}}

\newcommand{\sbs}[1]{_{\scriptscriptstyle \mathrm{#1}}}

\graphicspath{{figs/}}

\begin{document}
\begin{frontmatter}

\title{Structured Positive-Definite Optimal Control Synthesis of Closed-Loop Recommendation Systems over Social Networks%
\thanksref{footnoteinfo}}

\thanks[footnoteinfo]{This work has been supported in part by ANR projects Feeding Bias (ANR-22-CE380017-01) and MIAI Cluster BEAR (ANR-23-IACL-0006).}

\author[First,Second]{Simone Mariano}
\author[First]{Paolo Frasca}

\address[First]{Univ.\ Grenoble Alpes, CNRS, Inria, Grenoble INP, GIPSA-lab, 38000 Grenoble, France (e-mail: simone.mariano@grenoble-inp.fr, paolo.frasca@gipsa-lab.fr).}
\address[Second]{Univ.\ Grenoble Alpes, CNRS, Sciences Po Grenoble-UGA, Pacte, 38000 Grenoble, France.}

\begin{abstract}
We study the design of feedback recommendation policies for networked multi-topic opinion dynamics. The design of recommendations is formulated as an infinite-horizon linear-quadratic optimal control problem that favors suggestions aligned with each agent's current opinion, thereby using opinion alignment as a proxy for engagement. At the same time, the performance index penalizes polarization, deviation from the uncontrolled equilibrium of the opinion dynamics, and from the current agents' opinion, recommendation effort, and favors coherence among neighbors' recommendations. An agent-wise representation exposes the interconnection structure of the social network and enables structured controller and certificate synthesis. Under the assumption that the stage cost of the performance index is strictly positive definite, the affine optimal control problem separates into a strictly convex steady-state optimization, which determines the optimal equilibrium and the affine controller offset,  and an LQR problem, which yields the static state-feedback gain. The resulting centralized Riccati controller provides a performance benchmark, while structured gains are obtained from dissipativity-based and $\mathcal H_2$-based LMI surrogate formulations. \green{Possibly overlapping certificate clusters yield scalable local sufficient conditions for the structured gain design, while a separate steady-state cluster decomposition provides an exact consensus reformulation of the steady-state optimization.} Numerical results on a social network illustrate the closed-loop behavior and the trade-off between performance and controller locality.
\end{abstract}


\end{frontmatter}

%\tableofcontents

\section{Introduction}

Recommendation systems increasingly shape how users encounter information, products, and social media content. When recommendations are optimized mainly for short-term engagement, the feedback loop between platform decisions and user behavior may amplify confirmation bias, reduce content diversity, contribute to opinion polarization towards extreme opinions, and foster echo chambers~\citep{gausen2022abm,huszar2022algorithmic,bail2018opposing}. The results in~\citep{FrascaRecommendation2022}, which formalized this mechanism in a closed-loop single-user model, provide both a baseline mechanism and a benchmark phenomenon, while \citep{jiang2019degenerate} and \citep{chaney2018confounding} show, respectively, how without sufficient exploration, opinions drift to extremes, and how learning on self-influenced data may reduce diversity and degenerate social platforms.

The work \citep{dean2024accounting} surveyed the burgeoning literature documenting the shortcomings of favoring short-term engagement, and argues that most recommendation pipelines are designed without an explicit model of how users and algorithms shape one another and highlights several structural issues as (i) memoryless architectures \citep{covington2016youtube} (ii) simplified or omitted user dynamics and creator adaptation (iii) and optimization aggressively oriented toward engagement \citep{chen2019topk,immorlica2024clickbait}.  Mitigation is typically an afterthought, and performed favoring myopic, symptom-level fixes, which use a posteriori logged data \citep{sinha2016deconvolving,chaney2018confounding}. This, in turn, degrades platform performance even more and induces user- and society-level harms, including shifts in exposure composition, reduced content diversity, and polarization  \citep{mansoury2020feedback,nguyen2014exploring}.

Recommendation system design should therefore not be viewed only as a prediction or ranking algorithm, but also regarded as a feedback-control problem in which recommendations influence the user state on which future recommendations are based. The user-recommender feedback is studied in the present work by extending the results introduced by \citep{mariano_frasca_pathological}. We keep the modeling structure proposed therein and thus consider a networked multi-topic opinion dynamics model inspired by \citep{FrascaRecommendation2022,friedkin2016network,ye2020continuous}, where recommendations enter the user's model as control inputs. These recommendations are chosen to minimize a performance index, which rewards recommendations aligned with the current opinion of each agent, reflecting the tendency of users to interact with opinion-congruent content. At the same time, it penalizes polarization, deviation from the uncontrolled equilibrium, recommendation effort, disagreement between neighboring recommendations, and recommendation mismatch.

The quadratic reward associated with the engagement proxies makes well-posedness a nontrivial issue that is studied in \citep{mariano_frasca_companion}. The present paper instead concentrates on constructive controller synthesis in the strictly positive-definite, well-posed regime. Under this condition, the affine infinite-horizon problem separates into two standard optimization problems. A strictly convex static quadratic program determines the optimal steady-state pair, while an LQR problem determines the static state-feedback gain. The corresponding affine term in the controller is then fixed by the requirement that the closed-loop system converge to the optimal steady state.

The centralized LQR gain is generally dense and may require each recommendation to depend on the opinions of the entire network. To account for limited information availability, we consider affine feedback laws whose static gains satisfy a prescribed block-sparsity pattern. The resulting structured optimal control problem is nonconvex. We therefore develop two convex surrogate formulations based, respectively, on a dissipativity inequality and an $\mathcal H_2$ performance criterion.

We also develop \green{two-layer formulations} aimed at improving scalability. For the structured gain design, the centralized LMIs are decomposed into local matrix inequalities whose satisfaction is sufficient for the centralized certificate. This construction preserves the prescribed information pattern but may introduce conservatism. The steady-state quadratic program is instead decomposed exactly over a set of agents associated with the sparsity pattern of its hessian. Consensus constraints on the shared variables recover the centralized optimal steady state. Finally, the numerical illustrations show the complete design procedure on a social network with an emerging community structure. The centralized Riccati controller and the two structured controllers are compared in terms of their closed-loop trajectories, cumulative performance, individual cost components, and agent-wise information patterns. The results illustrate the trade-off between the locality of the feedback policy and its performance.

The closed-loop interaction between recommendation policies and evolving user opinions has been studied from several complementary perspectives. A broader control-theoretic perspective is developed in \citep{depasquale2026recommender}, where recommender systems are modeled as coupled dynamical systems involving users, creators, and the recommendation algorithm, and their fairness is analyzed according to their long-term closed-loop effects. The work in \citep{FrascaRecommendation2022} introduces a closed-loop model relating personalized recommendations, confirmation bias, and the opinion evolution of a single user. A networked extension is considered in \citep{sprenger2024control}, where scalar Friedkin--Johnsen dynamics are influenced through recommendation policies designed using model-free control and constrained model predictive control. More recently, \citep{chandrasekaran2026network} develops a network-aware recommender based on online data-driven feedback optimization. The proposed policy balances user engagement and polarization mitigation while relying on online click data and provides stability and optimality guarantees for a network governed by Friedkin--Johnsen dynamics. The present work differs from these contributions in both the model and the controller synthesis problem. More specifically, we study the synthesis of structured affine feedback policies for networked multi-topic opinion dynamics under a strictly positive-definite infinite-horizon optimal control formulation.

Other related approaches act on different variables of the opinion dynamics model. In particular, \citep{kuhne2025optimizing} optimizes the weights of the social-interaction network through a scalable hypergradient method, with the aim of reducing polarization and disagreement under Friedkin--Johnsen dynamics. Thus, the designed intervention modifies the influence network, rather than synthesizing an agent-wise recommendation input as a structured state-feedback law. Related works \citep{carroll2022preference,breschi2024optimal} also study recommendation-induced preference shifts or personalized policies, but do not address the structured infinite-horizon control problem considered here.

A preliminary version of this work and \citep{mariano_frasca_companion} appeared in our contribution~\citep{mariano_frasca_pathological}. The present manuscript complements that contribution, and that of  \citep{mariano_frasca_companion}, by developing the complete affine steady-state and LQR decomposition, the structured dissipativity-based and $\mathcal H_2$-based formulations, the certificate cluster synthesis conditions, and the exact steady-state cluster formulation of the steady-state problem.

The remainder of the paper is organized as follows. Section~\ref{sec:pi} introduces the performance index and the control problem. Section~\ref{sec:model} presents the opinion dynamics, the agent-wise representation, and the reduced affine LQ formulation, while Section~\ref{subsec:FE_PD} derives the centralized solution and formulates the exact structured problem. The dissipativity-based and $\mathcal H_2$-based structured synthesis methods are presented in Sections~\ref{subsec:affine_offsets_LMI_eval_dis} and~\ref{subsec:affine_offsets_LMI_eval_H2}, with the certificate cluster formulations for the gain design and the steady-state cluster formulation for the steady-state optimization being developed in Section~\ref{subsec:local_sdp_final}. Section~\ref{sec:numerical-experiments} presents the numerical study, while Section~\ref{sec:disc} ends the paper with conclusions and discussions on the presented results.

\noindent\textbf{Notation:}
The sets of real and natural numbers are denoted by $\reals$ and $\naturals$, respectively. For $n\in\naturals$, let $[n]:=\{1,\dots,n\}$ and let $\mathbbm{1}_n$ denote the $n$-dimensional vector of ones. The symbols $I_n$ and $O_n$ denote the identity and zero matrices of appropriate dimensions. For matrices $A$ and $B$ of equal dimensions, $A\odot B$ denotes their Hadamard product. For a square matrix $M$, $\sigma(M)$ denotes its spectrum, $\Tr(M)$ its trace, and $\Sym(M):=M+M^\top$. For symmetric matrices, the symbols $\succ0$, $\succeq0$, $\prec0$, and $\preceq0$ denote the usual Loewner order. The sets of real diagonal, diagonal positive-definite, and diagonal positive-semidefinite matrices of dimension $n$ are denoted by $\mathbb D^n$, $\mathbb D_{\succ0}^n$, and $\mathbb D_{\succeq0}^n$, respectively. The sets of real symmetric and symmetric positive-definite matrices of dimension $n$ are denoted by $\mathbb S^n$ and
$\mathbb S_{\succ0}^n$, respectively. For $W_{(\cdot)} = \diag\left( w_{1,(\cdot)},\dots,w_{nm,(\cdot)} \right)$, we define $\lambda_{M,(\cdot)} := \max_{i\in[nm]}w_{i,(\cdot)}, \qquad \lambda_{m,(\cdot)} := \min_{i\in[nm]}w_{i,(\cdot)}$. For a complex number $\lambda$, $\Re(\lambda)$ denotes its real part. The operators $\col(\cdot)$ and $\diag(\cdot)$ denote column stacking and block-diagonal concatenation, respectively.

\section{Performance index and design goal}
\label{sec:pi}

Consider $n\in\naturals$ agents holding opinions on $m\in\naturals$ topics. The agents are connected through a directed weighted graph $\mathcal G=(\mathcal V,\mathcal E)$ with Laplacian $L=\Delta-\Gamma$, where $\Gamma\in\reals^{n\times n}$ is the adjacency matrix and $\Delta=\diag(\Gamma\mathbbm{1}_n)$. Let $z(t)\in\reals^{nm}$ and $w(t)\in\reals^{nm}$ denote, respectively, the opinion vector and the recommendation input, stacked agent-wise as $z=(z_1,\dots,z_n)$, and $w=(w_1,\dots,w_n)$, with $z_i,w_i\in\reals^m$. Thus $z_{(i-1)m+k}$ and $w_{(i-1)m+k}$ denote the opinion and recommendation input of agent $i$ on topic $k$. The performance index follows the architecture studied in~\citep{mariano_frasca_companion} and is defined as
\begin{align}
J(z,w) &:= \int_0^\infty \Big( -J\sbs{EN} +J\sbs{P} +J\sbs{D} +J\sbs{EX} +J\sbs{F} +J\sbs{C} -J_\circ \Big)\,\dif t \nonumber \\
&= \int_0^\infty \bigl(\ell(z(t),w(t))-J_\circ\bigr)\,\dif t, \label{eq:J}
\end{align}
where the individual terms $J\sbs{P}:= z^\top W\sbs{P}z$, $J\sbs{D}:=(z-z_{\mathrm{eq}})^\top W\sbs{D}(z-z_{\mathrm{eq}})$, $J\sbs{EX}:= w^\top W\sbs{EX}w$, and $J\sbs{C}:= (z-w)^\top W\sbs{C}(z-w)$ respectively penalize polarization, deviation from the uncontrolled equilibrium, recommendation effort, and recommendation mismatch, while the term $J\sbs{EN} := z^\top W\sbs{EN}w$ rewards the recommendation and opinion alignment which is used as a proxy for engagement \citep{FrascaRecommendation2022}. The weights $W\sbs{EN},W\sbs{P},W\sbs{D}\in\mathbb D_{\succeq0}^{nm}$, $W\sbs{EX},W\sbs{C}\in\mathbb D_{\succ0}^{nm}$, and $\alpha_{\mathrm F}\ge0$ dictate the priority given to each aspect of the optimization, while the scalar $J_\circ$ subtracts the steady-state contribution of the stage cost, so that the infinite-horizon criterion is evaluated relative to the limit regime. Finally, let $\mathcal G_b$ be the symmetrized graph associated with $\Gamma_b:=(\Gamma+\Gamma^\top)/2$, and let $\Pi_{\mathcal{F}}=\{\mathcal F_1,\dots,\mathcal F_r\}$, with $r\in [n]$, be the \emph{exposure partition} of $\mathcal V$ defining the \emph{exposure groups}. We assume that the subgraph of $\mathcal G_b$ induced by each $\mathcal F_\ell$ is connected. For each $\ell\in[r]$, let $L_{b,\ell}$ denote the Laplacian of this induced subgraph. After reordering the agents so that vertices belonging to the same exposure group are contiguous, define $L_{\mathcal F}:=\diag(L_{b,1},\dots,L_{b,r})$ and set $L_u:=L_{\mathcal F}\otimes I_m$. Thus $J\sbs{F}:=\alpha_{\mathrm F} w^\top L_u w$ regularizes recommendation discrepancies within the selected sets of agents.

Hence, the design problem is
\begin{align}
\min_{w}\quad &J = \int_0^\infty \bigl(\ell(z(t),w(t))-J_\circ\bigr)\,\dif t, \nonumber \\
\mathrm{s.t.}\quad &\dot z=f(z,w), \qquad z(0)=z_0 . \label{eq:prob_full}
\end{align}
The role of the next section is to specify $f$ and rewrite~\eqref{eq:prob_full} as an affine LQ problem.

\section{Opinion dynamics model and reduced affine LQ formulation}
\label{sec:model}

Let $Z\in\reals^{m\times n}$ be the opinion matrix, where $Z_{ki}(t)$ is the opinion of agent $i\in[n]$ on topic $k\in[m]$. The continuous-time opinion dynamics are
\begin{equation}
\dot Z = (C-I_m)Z - Z(L+A_a)^\top + Z_\circ A_a + W-Z . \label{eq:model_matrix_form}
\end{equation}
Here $L$ is the social-graph Laplacian, $C\in\reals^{m\times m}$ is the inter-topic coupling matrix, $A_a\in\mathbb D_{\succ0}^n$ is the anchoring matrix, $Z_\circ\in\reals^{m\times n}$ collects the agents' inner beliefs, and $W\in\reals^{m\times n}$ is the recommendation input. The term $W-Z$ means that recommendations act in relative form, so that a recommendation already aligned with the current opinion does not create an additional change in the opinion.

\begin{assumption}
\label{ass:C}
Matrix $C$ is such that $c_{ii}\ge0$ for all $i\in[m]$, $|c_{ij}|\le1$ for all $i,j\in[m]$, and, given $A:=C-I_m$, zero is a semisimple eigenvalue of $A$ with multiplicity $p\ge1$, while every $\lambda\in\sigma(A)$ such that $\lambda\neq0$ satisfies $\Re(\lambda)<0$.
\end{assumption}

Assumption~\ref{ass:C} is consistent with \citep[Ass.~1]{ye2020continuous} and prevents instability of the uncontrolled topic dynamics. Equation~\eqref{eq:model_matrix_form} extends the models of~\citep{friedkin2016network,ye2020continuous} by adding a recommendation input.

Using the agent-wise vectorization introduced in Section~\ref{sec:pi}, the dynamics become
\begin{align*}
\dot z &= \Big[ I_n\otimes(C-I_m) - (L+A_a)\otimes I_m \Big]z \nonumber \\
&+ (A_a\otimes I_m)z_\circ + w-z =:A_{uz}z+d_z+w-z, \nonumber
\end{align*}
where
\begin{equation*}
A_{uz}:= I_n\otimes(C-I_m) - (L+A_a)\otimes I_m, \qquad d_z:=(A_a\otimes I_m)z_\circ . \label{eq:Az_dz_defs}
\end{equation*}
Equivalently,
\begin{equation}
\dot z=A_{cz}z+d_z+w, \qquad A_{cz}:=A_{uz}-I_{nm}. \label{eq:agent_controlled_affine}
\end{equation}

The factor $I_n\otimes(C-I_m)$ acts locally on the topics of each agent, whereas $L\otimes I_m$ couples different agents along the social graph while acting identically on all topics. Defining
\begin{equation*}
L_T:=I_m-C, \label{eq:topic_laplacian}
\end{equation*}
the uncontrolled drift can be written as
\begin{equation*}
A_{uz} = -\Big( I_n\otimes L_T + L\otimes I_m + A_a\otimes I_m \Big).
\end{equation*}
Thus $L_T$ is Laplacian-like in a spectral sense: by Assumption~\ref{ass:C}, it has a semisimple zero eigenvalue and all remaining eigenvalues in the open right half-plane. If, in addition, $C\mathbbm{1}_m=\mathbbm{1}_m$ and $C_{ij}\ge0$, then $L_T$ is a standard graph Laplacian on the topic graph. In this case, $-A_{uz}$ can be interpreted as a Cartesian-product-type diffusion between the social graph and the topic graph, augmented by the anchoring term $A_a\otimes I_m$. Under Assumption~\ref{ass:C}, the matrices, th matrices $A_{uz}$ and $A_{cz}$ are Hurwitz; see \citep[Prop. 2]{mariano_frasca_companion}. Consequently, the uncontrolled system
\begin{equation*}
\dot z=A_{uz}z+d_z
\end{equation*}
admits the unique globally exponentially stable equilibrium
\begin{equation*}
z_{\mathrm{eq}}=-A_{uz}^{-1}d_z .
\end{equation*}

We now turn to the LQ reformulation of~\eqref{eq:prob_full} by combining the performance index in \eqref{eq:J} with the agent-wise dynamics in \eqref{eq:agent_controlled_affine}. This step produces the affine LQ formulation used in the synthesis sections below.

Expanding the deviation penalty gives
\begin{equation*}
J\sbs{D} = z^\top W\sbs{D}z - 2z_{\mathrm{eq}}^\top W\sbs{D}z + \kappa_{\mathrm D}, \qquad \kappa_{\mathrm D} := z_{\mathrm{eq}}^\top W\sbs{D}z_{\mathrm{eq}}.
\end{equation*}
Since the criterion depends only on the difference $\ell-J_\circ$, replacing $\ell$ by $\ell-\kappa_{\mathrm D}$ and $J_\circ$ by $J_\circ-\kappa_{\mathrm D}$ leaves the optimization problem unchanged. We henceforth use the symbols $\ell$ and $J_\circ$ for these normalized quantities for the sake of simplicity. The normalized stage cost is therefore
\begin{equation*}
\ell(z,w) = z^\top Q z + 2\,z^\top N w + w^\top R w + 2\,c^\top z,
\end{equation*}
with
\begin{align*}
Q&=Q^{\top}=W\sbs{D}+W\sbs{P}+W\sbs{C} \succ 0, \\
N&=N^{\top}=-\tfrac{1}{2}\,W\sbs{EN}-W\sbs{C} \preceq 0, \\
R&=R^{\top}=W\sbs{EX}+\alpha\sbs{F}L_u +W\sbs{C}\succ 0,
\end{align*}
and $c=-W\sbs{D}z_{\mathrm{eq}}$.

Defining $s:=w+R^{-1}N z$, the control problem \eqref{eq:prob_full} is equivalent to
\begin{align}
\min_{s}\quad & \widetilde J=\int_{0}^{\infty} \bigl(\widetilde \ell(z(t),s(t))-J_\circ\bigr)\,\dif t, \nonumber \\
\mathrm{s.t.}\quad & \dot z = \widetilde A z + s + d_z,\quad z(0)=z_0, \label{eq:prob_full_alt}
\end{align}
with
\begin{equation*}
 \widetilde A:=A_{cz} - R^{-1}N, \quad \widetilde Q:=Q-N R^{-1}N,
\end{equation*}
and
\begin{align*}
\widetilde\ell(z,s)&=z^\top \widetilde Q z + s^\top R s + 2\,c^\top z, \\
w(z)&=s(z)-R^{-1}N z.
\end{align*}

The definiteness of the reduced quadratic part
\begin{equation*}
\widetilde\ell_{sq}(z,s) := z^\top \widetilde Q z+s^\top R s \label{eq:stage_cost_qp}
\end{equation*}
determines whether the reduced problem belongs to the classical strict LQ regime. The characterization of this regime is given in the companion paper~\citep{mariano_frasca_companion}. In particular, \citep[Lem.~3]{mariano_frasca_companion} provides a simple spectral sufficient condition, expressed directly in terms of the weights of the performance index, while \citep[Cor.~4]{mariano_frasca_companion} gives a necessary and sufficient condition when $Q$, $N$, and $R$ are simultaneously orthogonally diagonalizable. The latter case includes, for instance, jointly diagonalizable designs, and in particular the diagonal case obtained when the graph-regularization term is absent.
\begin{assumption}
\label{ass:strict_pd_reduced_cost}
The performance-index weights are chosen so that $\widetilde Q\succ0$.
\end{assumption}
\vspace{-0.9\baselineskip}
Since the input matrix in~\eqref{eq:prob_full_alt} is $I_{nm}$, the pair $(\widetilde A,I_{nm})$ is controllable. Therefore, under Assumption~\ref{ass:strict_pd_reduced_cost} and since $R \succ 0$, the reduced problem falls within the classical strictly positive-definite affine LQ setting. The affine terms $d_z$ and $c$ determine the optimal steady regime, while the transient deviation from that regime is governed by a standard stabilizing Riccati equation. Thus, the recommendation policy can be constructed as a stabilizing affine state feedback, as shown in the following sections.
\section{Centralized and structured affine LQ synthesis}
\label{subsec:FE_PD}

Under Assumption~\ref{ass:strict_pd_reduced_cost}, the reduced problem~\eqref{eq:prob_full_alt} is a strictly positive-definite affine LQ problem. The affine drift $d_z$ and the linear term $2 c^\top z$ prevent a direct application of the homogeneous LQR formula in the original coordinates. Hence, we shift the problem around the appropriate steady-state regime, which is also the point at which the offset $J_\circ$ is determined. Once this steady-state shift is performed, the gain synthesis reduces to the computation of the unique stabilizing solution of an algebraic Riccati equation. This construction provides the reference solution for the structured synthesis methods introduced later and makes explicit how the chosen design weights shape the optimal recommendation law and the resulting closed-loop behavior.

\subsection{Centralized affine LQ benchmark}
\label{subsec:FE_PD_agentwise_formulation}

Consider the control problem in \eqref{eq:prob_full_alt}, namely
\begin{align*}
\min_{s}\quad &\widetilde J = \int_{0}^{\infty} \bigl( \widetilde\ell(z(t),s(t))-J_\circ \bigr)\,\dif t, \\
\mathrm{s.t.}\quad &\dot z=\widetilde A z+s+d_z, \qquad z(0)=z_0.
\end{align*}
Following \citep{AmritRawlingsAngeli2011}, we can rewrite \eqref{eq:prob_full_alt} as
\begin{align}
\min_{s}\quad &\widetilde J := \int_{0}^{\infty} \bigl( \widetilde\ell(z(t),s(t)) - \widetilde\ell(z^\star,s^\star) \bigr)\,\dif t \nonumber \\
\mathrm{s.t.}\quad &\dot z=\widetilde A z+d_z+s, \qquad z(0)=z_0. \label{eq:agentwise_equiv_lq_inc}
\end{align}
with $J_\circ=\widetilde \ell(z^\star,s^\star)$ under the normalized convention introduced above, where $(z^\star,s^\star)$ is the unique minimizer of the steady-state program
\begin{align}
\label{eq:agentwise_equiv_lq_inc_min} \min_{z,s}\;&\widetilde \ell(z,s) \quad\text{s.t.}\quad 0=\widetilde A z + d_z + s.
\end{align}
Since $\widetilde Q\succ0$ and $R\succ0$, the objective is strictly convex and the constraint is affine. Hence the steady-state problem admits a unique optimizer $(z^\star,s^\star)$. Moreover, the equality constraint has full row rank because the input appears through the identity matrix. Therefore, the KKT conditions are necessary and sufficient, and there exists a unique multiplier $\lambda^\star$ such that the primal-dual triple $(z^\star,s^\star,\lambda^\star)$ satisfies the KKT conditions \citep{BoydVandenberghe2004}
\begin{align}
0 &= \widetilde A z^\star + d_z + s^\star, \nonumber \\
0 &= 2\widetilde Q z^\star + 2c + \widetilde A^\top\lambda^\star, \nonumber \\
0 &= 2R s^\star + \lambda^\star. \label{eq:KKT_conds}
\end{align}

Define the deviation variables
\begin{equation*}
\bar z:=z-z^\star,\qquad \bar s:=s-s^\star.
\end{equation*}
Since
\begin{equation*}
\widetilde \ell(z,s)=z^\top \widetilde Q z+s^\top R s+2c^\top z,
\end{equation*}
a direct expansion of the incremental integrand gives
\begin{align*}
\widetilde \ell(z,s)&-\widetilde \ell(z^\star,s^\star) = (z^\star+\bar z)^\top \widetilde Q (z^\star+\bar z) -(z^\star)^\top \widetilde Q z^\star \\
&\qquad \quad + (s^\star+\bar s)^\top R (s^\star+\bar s) -(s^\star)^\top R s^\star +2c^\top \bar z \\
&\qquad \quad = \bar z^\top \widetilde Q \bar z +\bar s^\top R \bar s +2\bar z^\top(\widetilde Q z^\star+c) +2\bar s^\top R s^\star .
\end{align*}
Using the KKT second and third conditions in \eqref{eq:KKT_conds}, we obtain
\begin{equation*}
2(\widetilde Q z^\star+c)=-\widetilde A^\top\lambda^\star, \qquad 2Rs^\star=-\lambda^\star,
\end{equation*}
and therefore
\begin{equation*}
\widetilde \ell(z,s)-\widetilde \ell(z^\star,s^\star) = \bar z^\top \widetilde Q \bar z +\bar s^\top R \bar s -\lambda^{\star\top}(\widetilde A\bar z+\bar s).
\end{equation*}
Since the deviation dynamics satisfy
\begin{equation*}
\dot{\bar z}=\widetilde A\bar z+\bar s,
\end{equation*}
it follows that
\begin{equation*}
\widetilde \ell(z,s)-\widetilde \ell(z^\star,s^\star) = \bar z^\top \widetilde Q \bar z +\bar s^\top R \bar s -\lambda^{\star\top}\dot{\bar z}.
\end{equation*}

Hence, integrating over $[0,T]$, with $T>0$, yields
\begin{align} \nonumber
&\int_{0}^{T}\Big(\widetilde \ell(z(t),s(t))-\widetilde \ell(z^\star,s^\star)\Big)\,\dif t
\\=
&\int_{0}^{T}\Big(\bar z^\top \widetilde Q\bar z + \bar s^\top R\bar s\Big)\,\dif t \nonumber\\
&\quad +(\lambda^\star)^\top \bar z(0)-(\lambda^\star)^\top \bar z(T).
\label{eq:Jinc_split_finite}
\end{align}

In this setting, and in line with the objective of the manuscript, we restrict attention to admissible inputs $s(\cdot)$ for which the corresponding deviation trajectory $\bar z(t):=z(t)-z^\star$ converges to zero as $t\to\infty$. On this class of trajectories, passing to the limit $T\to\infty$ in \eqref{eq:Jinc_split_finite} gives
\begin{equation*}
 \widetilde J = \int_{0}^{\infty}\Big(\bar z^\top \widetilde Q\bar z + \bar s^\top R\bar s\Big)\,\dif t +(\lambda^\star)^\top \bar z(0),
\end{equation*}
where the last term is constant with respect to the minimization.

Therefore, minimizing \eqref{eq:agentwise_equiv_lq_inc} over this admissible class is equivalent, in the sense of yielding the same optimal input, to the strict LQR
\begin{align}
\min_{\bar s(\cdot)}\;&\int_{0}^{\infty}\Big(\bar z(t)^\top \widetilde Q \bar z(t)+\bar s(t)^\top R \bar s(t)\Big)\,\dif t \nonumber \\
\text{s.t. }\;&\dot{\bar z}=\widetilde A\bar z+\bar s,\qquad \bar z(0)=z(0)-z^\star. \label{eq:std_lqr_bar2}
\end{align}

Let $P\succeq0$ be the stabilizing solution of the algebraic Riccati equation associated with $(\widetilde A,I_{nm},\widetilde Q,R)$. Then the optimal deviation input is
\begin{equation*}
\bar s(t)=-K\bar z(t),\qquad K=R^{-1}P,
\end{equation*}
and the corresponding optimal inputs are
\begin{align*}
s(t)=s^\star-K(z-z^\star), \quad w(z)=s^\star-K(z-z^\star)-R^{-1}Nz.
\end{align*}
The solution to the centralized unstructured problem is the exact benchmark in the strictly positive-definite regime. In general, however, the Riccati gain $K$ is dense, and the resulting recommendation law requires global state information. This is incompatible with locality, scalability, and limited-information requirements in large networks. We therefore turn to structured affine feedback laws whose gains are constrained to satisfy prescribed sparsity patterns.

\subsection{Structured affine feedback: sparsity constraint and exact (nonconvex) formulation}
\label{subsec:structured_affine}

The generally dense Riccati-based solution associated with \eqref{eq:std_lqr_bar2} is unsuitable for large-scale implementations in which only local information is available. To obtain a more realistic recommendation architecture, we constrain the feedback gain to satisfy a prescribed sparsity pattern compatible with the interconnection structure of the dynamics. In this way, the unconstrained optimal synthesis problem is replaced by a structured one.

Given the binary mask
$\mathcal M\in\{0,1\}^{nm\times nm}$ chosen to encode the desired interaction structure and typically aligned with the sparsity pattern of the dynamics, define the structured set
\begin{equation*}
\Sparse(\mathcal M)\!:=\!\Big\{K\!\in\!\reals^{nm\times nm}\!:\  (\mathbbm{1}_{nm\times nm}\!-\!\mathcal M)\!\odot\! K\!=\!0\Big\}.
\end{equation*}
We restrict to affine state feedback laws of the form
\begin{equation*}
s(z)=-K_s z-k_s,
\quad K_s\in\Sparse(\mathcal M),\; k_s\in\reals^{nm},
\end{equation*}
and require the closed-loop matrix
\begin{equation}\label{eq:structured_stab}
A_{\mathrm{cl}}(K_s):=\widetilde A-K_s,
\end{equation}
to be Hurwitz. Accordingly, $(K_s,k_s)$ denotes a generic admissible structured affine controller compatible with the mask $\mathcal M$.

Notice that, for a generic Hurwitz $A_{\mathrm{cl}}(K_s)$, the closed loop generally converges to an equilibrium
$(z_\infty,s_\infty)$ that depends on $(K_s,k_s)$ and does not necessarily coincide with $(z^\star,s^\star)$ stemming from \eqref{eq:agentwise_equiv_lq_inc_min}.
As a consequence, to obtain a finite infinite-horizon performance criterion under structured feedback, we measure excess cost
relative to the induced equilibrium as
\begin{equation*}
\widetilde J(K_s,k_s)
\!:=\!\int_0^\infty\!\Big(\widetilde \ell(z(t),s(t))- \widetilde \ell(z_\infty,s_\infty)\Big)dt.
\end{equation*}
Under \eqref{eq:structured_stab}, the equilibrium is uniquely defined by
\begin{align}\label{eq:structured_equilibrium}
0&=(\widetilde A-K_s)z_\infty+d_z-k_s, \nonumber\\
s_\infty&=-K_s z_\infty-k_s.
\end{align}

Define deviations $\tilde z:=z-z_\infty$ and $\tilde s:=s-s_\infty$. Then
\begin{equation*}
\dot{\tilde z}=(\widetilde A-K_s)\tilde z,\quad \tilde s=-K_s\tilde z,
\end{equation*}
and expanding $\widetilde \ell(z,s)-\widetilde \ell(z_\infty,s_\infty)$ yields
\begin{align*}
&\widetilde \ell(z,s)-\widetilde \ell(z_\infty,s_\infty) =
\tilde z^\top\!\big(\widetilde  Q+K_s^\top R K_s\big)\tilde z
+2\,\tilde z^\top g(K_s,k_s),
\end{align*}
with
\begin{equation*}
g(K_s,k_s):=\widetilde  Q z_\infty+c - K_s^\top R\, s_\infty .
\end{equation*}

Under \eqref{eq:structured_stab}, the deviation satisfies $\tilde z(t)=e^{(\widetilde A-K_s)t}\tilde z_0$, where $\tilde z_0:=z(0)-z_\infty$, so
\begin{align*}
\widetilde  J(K_s,k_s)&=\int_0^\infty \tilde z(t)^\top M(K_s)\tilde z(t)\,\dif t \\
&\quad+   2\int_0^\infty \tilde z(t)^\top g(K_s,k_s)\,\dif t,
\end{align*}
with $M(K_s):=\widetilde  Q+K_s^\top R K_s$.
By standard Lyapunov integral identities \citep[Ch.~3]{AndersonMoore1989}, the two integrals admit the representation
\begin{equation*}
\widetilde J (K_s,k_s)
=
\tilde z_0^\top S(K_s)\tilde z_0 + 2\,q(K_s,k_s)^\top \tilde z_0,
\end{equation*}
where $S(K_s)=S(K_s)^\top\succ 0$ is the unique solution of
\begin{align*}
(\widetilde A\!-\!K_s)^\top S &+ S(\widetilde A\!-\!K_s) \nonumber\\
&+ (\widetilde  Q+K_s^\top R K_s)=0,
\end{align*}
and $q(K_s,k_s)\in\reals^{nm}$ is the unique solution of
\begin{equation*}
(\widetilde A-K_s)^\top q + g(K_s,k_s)=0.
\end{equation*}
Consequently, the structured synthesis problem is posed, for a given initial condition $z(0)$, as
\begin{equation}\label{eq:structured_exact_opt}
\begin{aligned}
(K_s^\star,k_s^\star)\in \argmin_{K_s,k_s}\quad & \widetilde J (K_s,k_s)\\
\text{s.t.}\quad & K_s\in\Sparse(\mathcal M),\\
& k_s\in\reals^{nm},\\
& \widetilde A-K_s\ \text{Hurwitz}.
\end{aligned}
\end{equation}

Problem \eqref{eq:structured_exact_opt} is generally nonconvex and bilinear \citep{LinFardadJovanovic2011,RotkowitzLall2006,TokerOzbayACC1995,fardad2014design}. Moreover, the affine term depends on the closed-loop
equilibrium through \eqref{eq:structured_equilibrium}, which couples $(K_s,k_s)$.
These features make direct synthesis difficult at scale when $\mathcal M$ encodes locality.

This motivates replacing \eqref{eq:structured_exact_opt} by tractable surrogate problems expressed as LMIs. Lyapunov inequalities can certify stability and upper-bound the excess cost, and the LMI decision variables can be constrained to satisfy sparsity patterns compatible with $\mathcal M$. When the certificate itself is chosen to be structured, the feasibility and performance guarantees can be verified from local information, providing scalable, agent-wise certificates.

\section{Structured LMI synthesis and evaluation: a dissipativity approach}
\label{subsec:affine_offsets_LMI_eval_dis}

We now derive a tractable sufficient method for computing the structured static gain $K_s$. The exact problem \eqref{eq:structured_exact_opt} jointly optimizes $(K_s,k_s)$ and is generally nonconvex under the prescribed sparsity constraint. We therefore adopt a sequential design procedure. First, in the \emph{structured gain layer}, we synthesize $K_s$ by considering the homogeneous deviation dynamics and the quadratic transient contribution associated with $\widetilde Q+K_s^\top R K_s$. A dissipativity inequality is used to certify closed-loop stability and provide an upper bound on this contribution.

The optimal steady-state pair $(z^\star,s^\star)$ is determined separately from the steady-state optimization problem, in the \emph{steady-state layer}. Once the structured gain $K_s$ and the optimal steady-state pair have been computed, the affine offset $k_s$ is recovered so that the closed loop converges to $(z^\star,s^\star)$. With
\begin{equation*}
A_{\mathrm{cl}}:=\widetilde A-K_s,\qquad \tilde z:=z-z^\star,\qquad \tilde s:=s-s^\star,
\end{equation*}
the deviation dynamics are
\begin{equation}
\label{eq:incremental_homogeneous}
\dot{\tilde z}=A_{\mathrm{cl}}\tilde z,\qquad \tilde s=-K_s\tilde z.
\end{equation}
For an arbitrary initial deviation $\tilde z_0$, the corresponding quadratic transient cost is
\begin{equation*}
J_{\mathrm{tr}}(\tilde z_0;K_s):=\int_0^\infty\left(\tilde z(t)^\top\widetilde Q\tilde z(t)+\tilde s(t)^\top R\tilde s(t)\right)\,\dif t,
\end{equation*}
where $\tilde z(0)=\tilde z_0$. Substituting $\tilde s=-K_s\tilde z$ gives
\begin{equation*}
J_{\mathrm{tr}}(\tilde z_0;K_s)=\int_0^\infty\tilde z(t)^\top\left(\widetilde Q+K_s^\top RK_s\right)\tilde z(t)\,\dif t.
\end{equation*}
Whenever $A_{\mathrm{cl}}$ is Hurwitz, this cost admits the exact Lyapunov characterization
\begin{equation*}
J_{\mathrm{tr}}(\tilde z_0;K_s)=\tilde z_0^\top S(K_s)\tilde z_0,
\end{equation*}
where $S(K_s)=S(K_s)^\top\succ0$ is the unique solution of
\begin{equation*}
A_{\mathrm{cl}}^\top S+SA_{\mathrm{cl}}+\widetilde Q+K_s^\top RK_s=0.
\end{equation*}

Without the structural constraint, minimizing this quadratic cost for every initial condition leads to the standard infinite-horizon LQR solution in \eqref{eq:std_lqr_bar2}. Under the constraint $K_s\in\Sparse(\mathcal M)$, the corresponding structured synthesis problem is nonconvex. We therefore seek a structured stabilizing gain together with a quadratic certificate that bounds $J_{\mathrm{tr}}(\tilde z_0;K_s)$ for every initial deviation $\tilde z_0$. The resulting LMI formulation is independent of a particular initial condition.

Specifically, suppose there exists $P\succ0$ such that
\begin{equation}\label{eq:dissipation_ineq_core_final}
\frac{\dif}{\dif t}\bigl(\tilde z^\top P\tilde z\bigr)+\tilde z^\top\widetilde Q\tilde z+\tilde s^\top R\tilde s\le0
\qquad \forall\,\tilde z.
\end{equation}
Then $\tilde z^\top P\tilde z$ is a quadratic storage function for the deviation dynamics. In particular, $A_{\mathrm{cl}}$ is Hurwitz, and integrating \eqref{eq:dissipation_ineq_core_final} along \eqref{eq:incremental_homogeneous} yields
\begin{equation}
\label{eq:guaranteed_cost_bound_final}
J_{\mathrm{tr}}(\tilde z_0;K_s)\le \tilde z_0^\top P\tilde z_0,\qquad \forall\,\tilde z_0\in\reals^{nm}.
\end{equation}
A sufficient condition for \eqref{eq:dissipation_ineq_core_final} is
\begin{equation}\label{eq:quad_ineq_P_final}
A_{\mathrm{cl}}^\top P + P A_{\mathrm{cl}} + \widetilde Q + K_s^\top R K_s \prec 0.
\end{equation}

With the change of variables
\begin{equation*}
X:=P^{-1}\succ0,\qquad Y:=K_s X,
\end{equation*}
so that $K_s=Y X^{-1}$, condition \eqref{eq:quad_ineq_P_final} is implied, via Schur complement, by the LMI
\begin{equation}\label{eq:XY_LMI_guaranteed_cost_final}
\begin{bmatrix}
\Sym(\widetilde A X - Y) & X & Y^\top\\
X & -\widetilde Q^{-1} & 0\\
Y & 0 & -R^{-1}
\end{bmatrix}\prec 0.
\end{equation}

The LMI \eqref{eq:XY_LMI_guaranteed_cost_final} is therefore a convex sufficient condition for the existence of a structured stabilizing gain $K_s=Y X^{-1}$ together with the certified bound \eqref{eq:guaranteed_cost_bound_final}. To enforce locality, one may impose $Y\in\Sparse(\mathcal M)$ and, to preserve locality in the recovered gain, restrict $X\in\Sparse_C(\mathcal M)$ where
\begin{equation*}
\Sparse_C(\mathcal M)\!:=\!\Big\{X| Y X^{-1}\in \Sparse(\mathcal M), \forall Y \in \Sparse(\mathcal M) \Big\}
\end{equation*}
such that if $Y\in\Sparse(\mathcal M)$ then  $K_s=Y X^{-1} \in \Sparse(\mathcal M)$.

If one wishes to bias the design toward smaller certificates, and the structure-preserving constraint $X\in\Sparse_C(\mathcal M)$ is imposed through a convex representation, the guaranteed-cost synthesis can be posed as the semidefinite program
\begin{equation}\label{eq:LMI_trace_problem}
\begin{aligned}
\min_{X,Y,Z}\quad & \Tr(Z)\\
\text{s.t.}\quad & X=X^\top\succ0,\qquad Z=Z^\top\succeq0,\\
& \begin{bmatrix}
\Sym(\widetilde A X - Y) & X & Y^\top\\
X & -\widetilde Q^{-1} & 0\\
Y & 0 & -R^{-1}
\end{bmatrix}\prec 0,\\
& \begin{bmatrix} Z & I\\ I & X\end{bmatrix}\succeq 0,\\
& Y\in\Sparse(\mathcal M),\\
& X\in \Sparse_C(\mathcal M).
\end{aligned}
\end{equation}
The recovered gain is $K_s=Y X^{-1}$, while $Z\succeq X^{-1}=P$ provides a convex surrogate for the guaranteed-cost matrix. In particular, feasibility of \eqref{eq:LMI_trace_problem} guarantees the existence of $P=X^{-1}$ satisfying \eqref{eq:quad_ineq_P_final}, and thus stability together with the bound \eqref{eq:guaranteed_cost_bound_final}.

At this stage, the structured gain $K_s$ has been synthesized. We next compute the optimal steady-state pair $(z^\star,s^\star)$  by solving the same strictly convex steady-state problem \eqref{eq:agentwise_equiv_lq_inc_min} used for the centralized controller. This problem is unchanged by the previously selected gain $K_s$. Indeed, for a fixed $K_s$, an equilibrium of the affine controller satisfies
\begin{equation*}
0=(\widetilde A-K_s)z+d_z-k_s,\qquad s=-K_s z-k_s,
\end{equation*}
which gives
\begin{equation*}
s=-\widetilde A z-d_z,
\end{equation*}
which is precisely the constraint of \eqref{eq:agentwise_equiv_lq_inc_min} and does not depend on $K_s$. Once the optimal pair $(z^\star,s^\star)$ has been computed, the affine offset is recovered as
\begin{equation*}
k_s^\star=d_z+(\widetilde A-K_s)z^\star,
\end{equation*}
which guarantees
\begin{equation*}
z_\infty=z^\star,\qquad s_\infty=s^\star.
\end{equation*}
The local computation of $(z^\star,s^\star)$, and hence of $k_s^\star$, is developed in Section~\ref{subsec:local_sdp_final}.

The dissipativity-based construction provides a convex sufficient condition for structured stabilization together with an explicit upper bound on the quadratic transient cost. Its main potential limitation is conservatism: the storage matrix $P$ is optimized indirectly through an upper bound, and the imposed structure on the certificate may exclude otherwise admissible structured gains. We next consider an alternative synthesis method based on the LQR--$\mathcal H_2$ equivalence, which evaluates the same quadratic performance density through an averaged response measure.

\section{Structured LMI synthesis and evaluation: an $\mathcal H_2$-based approach}
\label{subsec:affine_offsets_LMI_eval_H2}

We now introduce a second route for synthesizing the structured gain. The dissipativity-based construction provides a quadratic upper bound on $J_{\mathrm{tr}}(\tilde z_0;K_s)$ for every initial deviation $\tilde z_0$ and uses a trace objective to obtain an initial-condition-independent surrogate. The $\mathcal H_2$-based formulation instead evaluates the aggregate quadratic response over the state directions selected by $B_0$, thereby providing a different initial-condition-independent criterion for the synthesis of $K_s$. Once the structured gain has been computed, the optimal steady-state pair $(z^\star,s^\star)$ is determined separately, and the affine offset is recovered so that the closed loop converges to this pair.

Consider the auxiliary homogeneous deviation dynamics
\begin{equation}
\label{eq:auxiliary_h2_deviation_system} \dot \eta = (\widetilde A-K_s)\eta + B_0 v,
\end{equation}
with performance output
\begin{equation}
\label{eq:h2_performance_output} \zeta = (C_1-D_{12}K_s)\eta,
\end{equation}
where
\begin{equation*}
C_1= \begin{bmatrix} \widetilde Q^{1/2} \\
0 \end{bmatrix}, \qquad D_{12}= \begin{bmatrix} 0 \\
R^{1/2} \end{bmatrix}.
\end{equation*}
Then
\begin{equation*}
\zeta^\top \zeta = \eta^\top \left( \widetilde Q+K_s^\top R K_s \right) \eta .
\end{equation*}
The matrix $B_0$ selects the state directions over which the transient performance is evaluated. The isotropic choice $B_0=I_{nm}$ penalizes the average response over all opinion directions, whereas lower-rank or weighted choices of $B_0$ can be used to emphasize selected agents or groups of them,  topics, or uncertain initial-condition subspaces.

For a fixed gain $K_s$ such that $A_{\mathrm{cl}}=\widetilde A-K_s$ is Hurwitz, the squared $\mathcal H_2$ norm of the map $v\mapsto\zeta$ can be bounded through a controllability Gramian $X=X^\top\succ0$ satisfying
\begin{equation*}
A_{\mathrm{cl}}X+XA_{\mathrm{cl}}^\top+B_0B_0^\top\prec0.
\end{equation*}
The associated output-energy contribution is
\begin{equation}
\operatorname{trace}\!\left((C_1-D_{12}K_s)X(C_1-D_{12}K_s)^\top\right). \label{eq:gram_ener}
\end{equation}
Thus, minimizing this quantity provides an initial-condition-agnostic surrogate for reducing the average quadratic transient response over the directions selected by $B_0$.

The gain $K_s$ is synthesized by minimizing an upper bound on the squared $\mathcal H_2$ norm of the map $v\mapsto \zeta$. Introducing the change of variables $Y=K_sX$ and imposing the structure-preserving constraint on $X$ through a convex representation, one obtains the convex synthesis condition
\begin{equation}
\label{eq:structured_h2_gain_synthesis} \begin{aligned} \min_{X,Y,Z}\quad & \operatorname{trace}(Z) \\
\text{s.t.}\quad & X=X^\top\succ0, \\
& \operatorname{Sym}(\widetilde A X-Y)+B_0B_0^\top \prec 0, \\
& \begin{bmatrix} Z & C_1X-D_{12}Y \\
(C_1X-D_{12}Y)^\top & X \end{bmatrix}\succeq0, \\
& Y\in\Sparse(\mathcal M), \\
& X\in\Sparse_C(\mathcal M), \end{aligned}
\end{equation}
where $Z$ is an auxiliary variable used to upper-bound the output-energy term \eqref{eq:gram_ener}. Indeed, by the Schur complement and the relation $Y=K_sX$, the third constraint implies
\begin{equation*}
Z\succeq(C_1-D_{12}K_s)X(C_1-D_{12}K_s)^\top.
\end{equation*}
As in the dissipativity-based formulation, restricting $X$ to the compatible structure-preserving class and imposing $Y\in\Sparse(\mathcal M)$ ensure that the recovered gain $K_s=YX^{-1}$ belongs to $\Sparse(\mathcal M)$.

Problem~\eqref{eq:structured_h2_gain_synthesis} should be interpreted as a structured transient-shaping surrogate. It does not solve the exact structured problem \eqref{eq:structured_exact_opt} for a prescribed initial condition. Instead, it produces a sparse stabilizing gain with a certified averaged quadratic transient-performance bound for the auxiliary system~\eqref{eq:auxiliary_h2_deviation_system}--\eqref{eq:h2_performance_output}.

Once $K_s$ has been obtained, the affine offset is recovered as
\begin{equation*}
 k_s^\star = d_z+(\widetilde A-K_s)z^\star.
\end{equation*}
This choice guarantees that the $\mathcal H_2$-based controller converges to the same optimal pair $(z^\star,s^\star)$ as the centralized benchmark and the dissipativity-based controller. Thus, the two-layer structure separates the synthesis of the structured static gain from the local computation of the fixed optimal steady state. The following section addresses both computations through local-to-global constructions.

\section{Local-to-global certificates for the two-layer design}
\label{subsec:local_sdp_final}

We now provide local constructions for the two-layer formulation, beginning with the structured gain layer. We construct cluster-based SDPs that use only local matrix blocks and local interaction information, while producing decision variables that satisfy the corresponding centralized gain-synthesis LMIs.

The central structural point is that the local decomposition must respect the block structure of the inverse terms $R^{-1}$ and $\widetilde Q^{-1}$, since both appear explicitly in the centralized LMIs. To this end, let
\begin{equation*}
\Pi_{\mathcal B}:=\{\mathcal B_1,\dots,\mathcal B_{q}\},
\end{equation*}
with $q\in [n]$, be a partition of $\mathcal V$ that is no finer than the finest block partition induced by $R$, referred to as the \emph{base-block partition}, and call each $\mathcal B_a$ a \emph{base block}. Define
\begin{equation*}
d_a:=|\mathcal B_a|m,\qquad a\in[q].
\end{equation*}
Equivalently, after ordering the agents according to $\Pi_{\mathcal B}$, one can write
\begin{align*}
R&=\diag(R_1,\dots,R_q)\succ0, \\
R^{-1}&=\diag(R_1^{-1},\dots,R_q^{-1}),
\end{align*}
with $R_a\in\mathbb S^{d_a}_{\succ0}$. In the present formulation, a natural choice for $\Pi_{\mathcal B}$ is the connected partition induced by the graph used in $J_{\mathrm F}$, or any coarsening of it. Since $\widetilde Q$ is diagonal, or agent-wise block diagonal, in the design considered here, it is also block diagonal with respect to $\Pi_{\mathcal B}$. Hence, with the same ordering,
\begin{align*}
\widetilde Q&=\diag(Q_1,\dots,Q_q)\succ0, \\
\widetilde Q^{-1}&=\diag(Q_1^{-1},\dots,Q_q^{-1}),
\end{align*}
with $Q_a\in\mathbb S^{d_a}_{\succ0}$.

We restrict the Lyapunov certificate to the same inverse-preserving block structure,
\begin{equation*}
X=\diag(X_1,\dots,X_q), \qquad X_a\in\mathbb S^{d_a}_{\succ0}.
\end{equation*}
Thus $X^{-1}$ is block diagonal with respect to $\Pi_{\mathcal B}$.

The sparsity constraint $Y\in\Sparse(\mathcal M)$ is interpreted at the resolution induced by $\Pi_{\mathcal B}$. Thus, the mask $\mathcal M$ specifies which blocks $Y_{ab}\in\mathbb R^{d_a\times d_b}$, associated with the pair $(\mathcal B_a,\mathcal B_b)$, are allowed to be nonzero. If $\Pi_{\mathcal B}$ is the agent-wise partition, this coincides with the usual $m\times m$ agent-level sparsity pattern. If $\Pi_{\mathcal B}$ is coarser, then $\mathcal M$ represents an aggregated inter-block sparsity pattern rather than a fine agent-level sparsity pattern.

With this convention, if $Y\in\Sparse(\mathcal M)$, then the recovered gain
\begin{equation*}
K_s=YX^{-1}
\end{equation*}
satisfies
\begin{equation*}
(K_s)_{ab}=Y_{ab}X_b^{-1}.
\end{equation*}
Hence $(K_s)_{ab}=0$ whenever $Y_{ab}=0$, so the prescribed inter-block sparsity pattern is preserved. Notice that, when $\Pi_{\mathcal B}$ is coarser than the agent-wise partition, the block $(K_s)_{ab}$ is generally dense whenever $Y_{ab}\neq0$, because $X_b^{-1}$ is generally dense inside $\mathcal B_b$. 
% Therefore, the present formulation preserves sparsity at the $\Pi_{\mathcal B}$-block level, not necessarily at the finer agent level inside each block.

For each base block $\mathcal B_a$, let
\begin{equation*}
E_a\in\mathbb R^{nm\times d_a}
\end{equation*}
be the selector such that $z_a=E_a^\top z$ collects the states of the agents in $\mathcal B_a$, with
\begin{equation*}
E_a^\top E_b=\delta_{ab}I_{d_a}.
\end{equation*}
We partition the global matrices according to $\Pi_{\mathcal B}$ as
\begin{equation*}
\widetilde A=[A_{ab}]_{a,b=1}^q, \qquad Y=[Y_{ab}]_{a,b=1}^q,
\end{equation*}
with
\begin{equation*}
A_{ab},Y_{ab}\in\mathbb R^{d_a\times d_b}.
\end{equation*}

Define the \emph{base-block interaction graph}
\begin{equation*}
\mathcal G_\mathcal B=([q],\mathcal E_\mathcal B)
\end{equation*}
as the undirected graph over the base blocks such that, for $a\neq b$, the edge $\{a,b\}$ belongs to $\mathcal E_\mathcal B$ whenever (i) $A_{ab}\neq0$ or $A_{ba}\neq0$, or (ii) the feedback blocks $Y_{ab}$ or $Y_{ba}$ are allowed by the prescribed block-level sparsity mask
% Thus, $\mathcal E_\mathcal B$ collects all off-diagonal block interactions that must be represented in the local certificates, either because they appear in the dynamic or because they are admissible in the structured feedback gain. Diagonal block contributions are not represented as edges and they are distributed separately among the clusters containing the corresponding block through the weights $\beta_{a\ell}$ below.
and choose a family of $p$ possibly overlapping \emph{certificate clusters}
\begin{equation}
\{\mathcal C_\ell\}_{\ell=1}^p, \qquad \mathcal C_\ell\subseteq[q], \label{eq:certC}
\end{equation}
 \green{such that every base block index belongs to at least one certificate cluster and, for every edge $\{a,b\}\in\mathcal E_\mathcal B$, there exists at least one $\ell\in[p]$ such that $a,b\in\mathcal C_\ell$.}  For each edge $\{a,b\}\in\mathcal E_\mathcal B$, \green{define the index set of certificate clusters containing both endpoints of that edge as}
\begin{equation*}
\mathcal I_{ab} := \{\ell\in[p]: a,b\in\mathcal C_\ell\}.
\end{equation*}
By construction, $\mathcal I_{ab}\neq\emptyset$ for every $\{a,b\}\in\mathcal E_\mathcal B$. We choose \emph{edge weights} $\theta_{ab,\ell}\ge0$, for $\ell\in\mathcal I_{ab}$, satisfying
\begin{equation}
\label{eq:theta_partition_unity_local_sdp} \sum_{\ell\in\mathcal I_{ab}}\theta_{ab,\ell}=1, \qquad \{a,b\}\in\mathcal E_\mathcal B,
\end{equation}
and the \emph{diagonal weights} $\beta_{a\ell}>0$, for $a\in\mathcal C_\ell$, below, satisfying
\begin{equation}
\label{eq:beta_partition_unity_local_sdp} \sum_{\ell:\,a\in\mathcal C_\ell}\beta_{a\ell}=1, \qquad a\in[q].
\end{equation}
\green{The weights $\beta_{a\ell}$ distribute the diagonal contribution associated with base-block index $a$ among the certificate clusters satisfying $a\in\mathcal C_\ell$, whereas the weights $\theta_{ab,\ell}$ distribute the off-diagonal interaction associated with $\{a,b\}$ among the local certificates indexed by $\ell\in\mathcal I_{ab}$.}

For a certificate cluster $\mathcal C_\ell=\{a_1,\dots,a_{|\mathcal C_\ell|}\}$, define
\begin{equation*}
E_{\mathcal C_\ell} := \begin{bmatrix} E_{a_1}&\cdots&E_{a_{|\mathcal C_\ell|}} \end{bmatrix}, \quad F_{\mathcal C_\ell} := \diag(E_{\mathcal C_\ell},E_{\mathcal C_\ell},E_{\mathcal C_\ell}),
\end{equation*}
\begin{equation*}
G_{\mathcal C_\ell}:=\diag(E_{\mathcal C_\ell},E_{\mathcal C_\ell}), \qquad H_{\mathcal C_\ell}:=\diag(G_{\mathcal C_\ell},E_{\mathcal C_\ell}),
\end{equation*}

and let $X_\ell^{\mathrm A}$ be the block-diagonal matrix formed by the blocks $X_a$, $a\in\mathcal C_\ell$ stacked in the appropriate ordering. Let also $X_\ell^\beta$, $Q_\ell^{-1}$, and $R_\ell^{-1}$ be the corresponding weighted block-diagonal matrices with diagonal blocks $\beta_{a\ell}X_a$, $\beta_{a\ell}Q_a^{-1}$, and $\beta_{a\ell}R_a^{-1}$, respectively. Define local block matrices $A_\ell=[(A_\ell)_{ab}]$ and $Y_\ell=[(Y_\ell)_{ab}]$, indexed by $a,b\in\mathcal C_\ell$, as follows. For diagonal blocks,
\begin{equation*}
(A_\ell)_{aa}=\beta_{a\ell}A_{aa}, \quad (Y_\ell)_{aa}=\beta_{a\ell}Y_{aa}.
\end{equation*}
For off-diagonal blocks, with $a\neq b$, define
\begin{align*}
&(A_\ell)_{ab} = \begin{cases} \theta_{ab,\ell}A_{ab}, & \{a,b\}\in\mathcal E_\mathcal B,\ \ell\in\mathcal I_{ab}, \\
0, & \text{otherwise}, \end{cases} \\
&(Y_\ell)_{ab} = \begin{cases} \theta_{ab,\ell}Y_{ab}, & \{a,b\}\in\mathcal E_\mathcal B,\ \ell\in\mathcal I_{ab}, \\
0, & \text{otherwise}. \end{cases}
\end{align*}
% Here $\beta_{a\ell}$ distributes diagonal block contributions among all clusters containing block $a$, whereas $\theta_{ab,\ell}$ distributes the off-diagonal interaction associated with the edge $\{a,b\}$ among all clusters containing that edge. 
Thus, the local matrices are not merely principal submatrices of the centralized LMI, but they are weighted local contributions whose lifted sum reconstructs the centralized matrix.

\begin{remark}
\green{The base-block partition $\Pi_{\mathcal B}=\{\mathcal B_1,\dots,\mathcal B_q\}$ and the certificate clusters $\{\mathcal C_\ell\}_{\ell=1}^p$ play different roles. The base blocks $\mathcal B_a$ define the resolution at which $R^{-1}$, $\widetilde Q^{-1}$, and $X^{-1}$ are locally representable. Each certificate cluster $\mathcal C_\ell\subseteq[q]$ identifies the collection of base blocks included in the corresponding local LMI and determines which diagonal and off-diagonal contributions are assigned to that local certificate.} If $\alpha_{\mathrm F}=0$, or more generally if $R$ is agent-wise block diagonal, one may take $\Pi_{\mathcal B}=\{\{1\},\dots,\{n\}\}$. When $J_{\mathrm F}$ is active, $R$ is generally block diagonal only with respect to the connected partition used in $L_{\mathcal F}$, and the blocks $\mathcal B_a$ should be chosen accordingly.
\end{remark}

\subsection{Local-to-global dissipativity-based design}
We first state the lifted-sum identity for the guaranteed-cost LMI.

\begin{proposition}
\label{prop:liftsum_assembly}
\green{Let $\{\mathcal C_\ell\}_{\ell=1}^p$ be the certificate-cluster family in \eqref{eq:certC}, and} assume that conditions \eqref{eq:theta_partition_unity_local_sdp}--\eqref{eq:beta_partition_unity_local_sdp} hold. Let
\begin{equation*}
\mathcal L(X,Y):= \begin{bmatrix} \Sym(\widetilde A X-Y) & X & Y^\top \\
X & -\widetilde Q^{-1} & 0 \\
Y & 0 & -R^{-1} \end{bmatrix},
\end{equation*}
and, for each $\ell\in[p]$, define
\begin{equation*}
 \mathcal L_\ell(X,Y):= \begin{bmatrix} \Lambda_\ell & X_\ell^\beta & Y_\ell^\top \\
X_\ell^\beta & -Q_\ell^{-1} & 0 \\
Y_\ell & 0 & -R_\ell^{-1} \end{bmatrix},
\end{equation*}
where
\begin{equation*}
\Lambda_\ell:=\Sym(A_\ell X_\ell^{\mathrm A}-Y_\ell).
\end{equation*}
Then
\begin{equation}
\label{eq:liftsum_L} \mathcal L(X,Y) = \sum_{\ell=1}^p F_{\mathcal C_\ell} \mathcal L_\ell(X,Y) F_{\mathcal C_\ell}^\top .
\end{equation}
Consequently, if $\mathcal L_\ell(X,Y)\preceq-\varepsilon_\ell I$ for all $\ell \in [p]$, with $\varepsilon_\ell>0$, then $\mathcal L(X,Y)\prec0$.
\end{proposition}

\begin{proof}
The proof is blockwise. Since
\begin{equation*}
F_{\mathcal C_\ell} = \diag(E_{\mathcal C_\ell},E_{\mathcal C_\ell},E_{\mathcal C_\ell}),
\end{equation*}
it suffices to verify the corresponding lifted identities at the level of the base blocks.

For the $(1,2)$ block,
\begin{equation*}
E_{\mathcal C_\ell}X_\ell^\beta E_{\mathcal C_\ell}^\top = \sum_{a\in\mathcal C_\ell} \beta_{a\ell}E_aX_aE_a^\top .
\end{equation*}
Summing over $\ell$ and using \eqref{eq:beta_partition_unity_local_sdp} gives $X$. The same calculation gives the $(2,2)$ and $(3,3)$ blocks, namely $-\widetilde Q^{-1}$ and $-R^{-1}$.

For the $Y$-blocks, the diagonal terms satisfy
\begin{equation*}
\sum_{\ell:\,a\in\mathcal C_\ell}\beta_{a\ell}Y_{aa}=Y_{aa}.
\end{equation*}
The off-diagonal terms are reproduced because, for each edge $\{a,b\}\in\mathcal E_\mathcal B$, the local contributions are weighted by $\theta_{ab,\ell}$ over all $\ell\in\mathcal I_{ab}$, and the weights satisfy
\begin{equation*}
\sum_{\ell\in\mathcal I_{ab}}\theta_{ab,\ell}=1.
\end{equation*}
The same argument applies to the off-diagonal blocks of $\widetilde A X$, since $X$ is block diagonal and
\begin{equation*}
(\widetilde A X)_{ab}=A_{ab}X_b .
\end{equation*}
The diagonal terms of $\widetilde A X$ sum correctly by \eqref{eq:beta_partition_unity_local_sdp}. This proves \eqref{eq:liftsum_L}.

Finally, if $\mathcal L_\ell\preceq-\varepsilon_\ell I$, then
\begin{equation*}
\mathcal L(X,Y) \preceq -\sum_{\ell=1}^p \varepsilon_\ell F_{\mathcal C_\ell}F_{\mathcal C_\ell}^\top .
\end{equation*}
Since every base block index belongs to at least one certificate cluster, the matrix on the right-hand side is negative definite. Hence $\mathcal L(X,Y)\prec0$.
\end{proof}

For fixed constants $\varepsilon_\ell>0$, the corresponding local guaranteed-cost SDP is
\begin{equation*}
 \begin{aligned} \min_{\{X_a\},\{Y_{ab}\},\{Z_a\}}\quad & \sum_{a=1}^q \Tr(Z_a) \\
\text{s.t.}\quad & X_a\succ0,\qquad a\in[q], \\
& Y\in\Sparse(\mathcal M), \\
& \mathcal L_\ell(X,Y)\preceq-\varepsilon_\ell I, \qquad \ell\in[p], \\
& \begin{bmatrix} Z_a & I \\
I & X_a \end{bmatrix}\succeq0, \qquad a\in[q]. \end{aligned}
\end{equation*}
By Proposition~\ref{prop:liftsum_assembly}, any feasible solution satisfies the centralized guaranteed-cost LMI. Moreover, by Schur complement, $Z_a\succeq X_a^{-1}$, so the objective upper-bounds $\Tr(X^{-1})$. The global certificate matrix is $X=\diag(X_1,\dots,X_q)$, and the structured gain is recovered as
\begin{equation*}
K_s=YX^{-1}.
\end{equation*}
Since $X^{-1}$ is block diagonal with respect to $\Pi_{\mathcal B}$, the sparsity pattern imposed on $Y$ is preserved in the recovered gain $K_s$.

\subsection{Local-to-global $\mathcal H_2$-based design}
We now give the corresponding $\mathcal H_2$-based certificates. Let
\begin{equation*}
W_0:=B_0B_0^\top .
\end{equation*}
For a purely local lifted-sum decomposition, choose $B_0$ so that $W_0$ is block diagonal with respect to $\Pi_{\mathcal B}$:
\begin{equation*}
W_0=\diag(W_{0,1},\dots,W_{0,q}), \qquad W_{0,a}\succeq0 .
\end{equation*}
The all-directions choice $B_0=I_{nm}$ gives $W_{0,a}=I_{d_a}$. Define the centralized matrices
\begin{equation*}
C_1= \begin{bmatrix} \widetilde Q^{1/2} \\
0 \end{bmatrix}, \qquad D_{12}= \begin{bmatrix} 0 \\
R^{1/2} \end{bmatrix}.
\end{equation*}
The centralized averaged $\mathcal H_2$-type certificate is given by
\begin{align}
\mathcal N(X,Y) &:= \Sym(\widetilde A X-Y)+W_0 \prec0, \label{eq:central_h2_N} \\
\mathcal O(X,Y,Z) &:= \begin{bmatrix} Z & C_1X-D_{12}Y \\
(C_1X-D_{12}Y)^\top & X \end{bmatrix} \succeq0, \label{eq:central_h2_O}
\end{align}
where
\begin{equation*}
Z= \diag (Z_1^{\mathrm Q},\dots,Z_q^{\mathrm Q}, Z_1^{\mathrm R},\dots,Z_q^{\mathrm R}).
\end{equation*}

For each certificate cluster $\mathcal C_\ell$, define
\begin{equation*}
W_{0,\ell}^{\beta} := \diag(\{\beta_{a\ell}W_{0,a}\}_{a\in\mathcal C_\ell})
\end{equation*}
and
\begin{equation*}
 \mathcal N_\ell(X,Y) := \Sym(A_\ell X_\ell^{\mathrm A}-Y_\ell) + W_{0,\ell}^{\beta}.
\end{equation*}
For the output LMI, define
\begin{align*}
&C_{1,\ell}^{\beta} := \begin{bmatrix} \diag(\{\beta_{a\ell}Q_a^{1/2}\}_{a\in\mathcal C_\ell}) \\
0 \end{bmatrix}, \\
&D_{12,\ell} := \begin{bmatrix} 0 \\
\diag(\{R_a^{1/2}\}_{a\in\mathcal C_\ell}) \end{bmatrix}.
\end{align*}
The factor $\beta_{a\ell}$, rather than $\sqrt{\beta_{a\ell}}$, is used in $C_{1,\ell}^{\beta}$ because the lifted local matrices are summed linearly. Define
\begin{equation*}
M_\ell(X,Y) := C_{1,\ell}^{\beta}X_\ell^{\mathrm A} - D_{12,\ell}Y_\ell .
\end{equation*}
Finally, introduce local output blocks by setting
\begin{equation*}
Z_\ell^\beta := \diag \left( \{\beta_{a\ell}Z_a^{\mathrm Q}\}_{a\in\mathcal C_\ell}, \{\beta_{a\ell}Z_a^{\mathrm R}\}_{a\in\mathcal C_\ell} \right),
\end{equation*}
and define
\begin{equation*}
 \mathcal O_\ell(X,Y,Z) := \begin{bmatrix} Z_\ell^\beta & M_\ell(X,Y) \\
M_\ell(X,Y)^\top & X_\ell^\beta \end{bmatrix}.
\end{equation*}

\begin{proposition}
\label{prop:h2_liftsum_assembly}
\green{Let $\{\mathcal C_\ell\}_{\ell=1}^p$ be the certificate-cluster family in \eqref{eq:certC}, and} assume that conditions \eqref{eq:theta_partition_unity_local_sdp}--\eqref{eq:beta_partition_unity_local_sdp} hold. Then the local $\mathcal H_2$-type matrices satisfy the lifted-sum identities
\begin{equation*}
\mathcal N(X,Y) = \sum_{\ell=1}^p E_{\mathcal C_\ell} \mathcal N_\ell(X,Y) E_{\mathcal C_\ell}^\top
\end{equation*}
and
\begin{equation*}
\mathcal O(X,Y,Z) = \sum_{\ell=1}^p H_{\mathcal C_\ell} \mathcal O_\ell(X,Y,Z) H_{\mathcal C_\ell}^\top .
\end{equation*}
Consequently, if
\begin{equation*}
\mathcal N_\ell(X,Y)\preceq-\varepsilon_\ell I, \qquad \mathcal O_\ell(X,Y,Z)\succeq0, \qquad \ell=1,\dots,p,
\end{equation*}
with $\varepsilon_\ell>0$, then the centralized constraints \eqref{eq:central_h2_N}--\eqref{eq:central_h2_O} hold.
\end{proposition}

\begin{proof}
The identity for $\mathcal N(X,Y)$ follows from the same blockwise argument used in Proposition~\ref{prop:liftsum_assembly}. The diagonal terms of $\Sym(\widetilde A X-Y)$ and $W_0$ are split according to the weights $\beta_{a\ell}$. The off-diagonal terms of $\Sym(\widetilde A X-Y)$ are reconstructed because, for each edge $\{a,b\}\in\mathcal E_\mathcal B$, the corresponding local contributions are weighted by $\theta_{ab,\ell}$ over all clusters containing that edge, and these weights satisfy \eqref{eq:theta_partition_unity_local_sdp}.

For the output LMI, the lower-right block gives
\begin{equation*}
\sum_{\ell=1}^p E_{\mathcal C_\ell} X_\ell^\beta E_{\mathcal C_\ell}^\top = X
\end{equation*}
by the partition-of-unity condition. The upper-left block gives
\begin{equation*}
\sum_{\ell=1}^p G_{\mathcal C_\ell} Z_\ell^\beta G_{\mathcal C_\ell}^\top = Z
\end{equation*}
for the same reason. It remains to check the off-diagonal block. The $\widetilde Q^{1/2}X$ part is diagonal in the base block partition and is reconstructed by the weights $\beta_{a\ell}$. The $R^{1/2}Y$ part is reconstructed because the diagonal blocks of $Y$ are split by the weights $\beta_{a\ell}$, whereas each admissible off-diagonal block of $Y$ is split over the local certificates indexed by $\ell\in\mathcal I_{ab}$ by the weights $\theta_{ab,\ell}$. Hence
\begin{equation*}
\sum_{\ell=1}^p G_{\mathcal C_\ell} M_\ell(X,Y) E_{\mathcal C_\ell}^\top = C_1X-D_{12}Y .
\end{equation*}
Therefore the lifted sum of the local output LMIs equals $\mathcal O(X,Y,Z)$. The implication follows immediately from summing the local semidefinite inequalities.
\end{proof}

For fixed constants $\varepsilon_\ell>0$, the associated local $\mathcal H_2$-type SDP is
\begin{equation}
\label{eq:local_SDP_h2_final} \begin{aligned} \min_{\{X_a\},\{Y_{ab}\},\{Z_a^{\mathrm Q}\},\{Z_a^{\mathrm R}\}}\quad & \sum_{a=1}^q\Tr(Z_a^{\mathrm Q}) + \sum_{a=1}^q\Tr(Z_a^{\mathrm R}) \\
\text{s.t.}\quad & X_a\succ0,\qquad a\in[q], \\
& Y\in\Sparse(\mathcal M), \\
& \mathcal N_\ell(X,Y)\preceq-\varepsilon_\ell I, \qquad \ell\in[p], \\
& \mathcal O_\ell(X,Y,Z)\succeq0, \qquad \ell\in[p]. \end{aligned}
\end{equation}
The global certificate matrix is $X=\diag(X_1,\dots,X_q)$, and the recovered gain is again
\begin{equation*}
K_s=YX^{-1}.
\end{equation*}
By Proposition~\ref{prop:h2_liftsum_assembly}, every feasible point of \eqref{eq:local_SDP_h2_final} satisfies the centralized averaged $\mathcal H_2$-type gain-synthesis LMIs.

\subsection{Steady-state layer}
We now turn to the steady-state layer. The objective is to compute the unique optimal steady-state pair $(z^\star,s^\star)$ of \eqref{eq:agentwise_equiv_lq_inc_min} from local objective contributions. This pair is fixed by the original normalized stage cost, and it is not selected through an additional designer-chosen surrogate. The steady-state sparsity graph is determined by the exact reduced static objective and \green{need not coincide with the base-block interaction graph $\mathcal G_{\mathcal B}$ used in the structured gain layer}. 

With the convention
\begin{equation*}
s(z)=-K_s z-k_s,
\end{equation*}
the closed-loop dynamics are
\begin{equation*}
\dot z=(\widetilde A-K_s)z+d_z-k_s.
\end{equation*}
Once $z^\star$ has been computed, the choice
\begin{equation*}
 k_s^\star = d_z+(\widetilde A-K_s)z^\star
\end{equation*}
enforces
\begin{equation*}
z_\infty=z^\star, \qquad s_\infty = -K_s z^\star-k_s^\star = -\widetilde A z^\star-d_z = s^\star.
\end{equation*}

Substituting the steady-state input through $s=-\widetilde A z-d_z$, the exact steady-state problem \eqref{eq:agentwise_equiv_lq_inc_min} is equivalent to
\begin{equation}
\label{eq:ss_static_problem_local_certificates} z^\star = \argmin_{z\in\mathbb R^{nm}} \Phi(z), \qquad \Phi(z) := \widetilde\ell \left( z,-\widetilde A z-d_z \right).
\end{equation}
Expanding the normalized stage cost gives
\begin{equation*}
\Phi(z) = \frac{1}{2} z^\top H_{\mathrm{ss}}z + h_{\mathrm{ss}}^\top z + \kappa_{\mathrm{ss}},
\end{equation*}
where
\begin{align*}
H_{\mathrm{ss}} &= 2\left( \widetilde Q + \widetilde A^\top R\widetilde A \right), \\
h_{\mathrm{ss}} &= 2\left( c + \widetilde A^\top Rd_z \right), \\
\kappa_{\mathrm{ss}} &= d_z^\top Rd_z.
\end{align*}

Since $\widetilde Q\succ0$ and $R\succ0$, $H_{\mathrm{ss}} = 2\left( \widetilde Q+\widetilde A^\top R\widetilde A \right) \succ0$.
% \begin{equation*}
% H_{\mathrm{ss}} = 2\left( \widetilde Q+\widetilde A^\top R\widetilde A \right) \succ0.
% \end{equation*}
Hence \eqref{eq:ss_static_problem_local_certificates} has the unique solution
\begin{equation*}
z^\star = -H_{\mathrm{ss}}^{-1}h_{\mathrm{ss}}, \qquad s^\star = -\widetilde A z^\star-d_z.
\end{equation*}
Notice that even if $\widetilde A$ is sparse and $R$ is block diagonal, the product $\widetilde A^\top R\widetilde A$ can introduce two-hop couplings. Hence the graph relevant for the steady-state layer is the steady-state sparsity graph associated to $H_{\mathrm{ss}}$, not necessarily the graph used in the structured gain layer.

We partition the matrices and vectors according to the same partition $\Pi_{\mathcal B}=\{\mathcal B_1,\dots,\mathcal B_q\}$ used in the structured gain layer,
\begin{equation*}
\widetilde A=[A_{ab}]_{a,b=1}^q, \, \widetilde Q=[\widetilde Q_{ab}]_{a,b=1}^q, \, R=\diag(R_1,\dots,R_q),
\end{equation*}
and
\begin{equation*}
c=\col(c_1,\dots,c_q), \qquad d_z=\col(d_{z,1},\dots,d_{z,q}).
\end{equation*}
Then, the steady-state hessian and linear term are partitioned conformably as
\begin{equation*}
H_{\mathrm{ss}}=[H_{ab}^{\mathrm{ss}}]_{a,b=1}^q, \qquad h_{\mathrm{ss}}=\col(h_1^{\mathrm{ss}},\dots,h_q^{\mathrm{ss}}).
\end{equation*}
Since
\begin{equation*}
H_{\mathrm{ss}} = 2\left( \widetilde Q+\widetilde A^\top R\widetilde A \right), \qquad h_{\mathrm{ss}} = 2\left( c+\widetilde A^\top Rd_z \right),
\end{equation*}
their blocks satisfy
\begin{equation*}
 H_{ab}^{\mathrm{ss}} = 2\left( \widetilde Q_{ab} + \sum_{k=1}^q A_{ka}^\top R_kA_{kb} \right), \qquad a,b\in[q],
\end{equation*}
and
\begin{equation*}
 h_a^{\mathrm{ss}} = 2\left( c_a+ \sum_{k=1}^q A_{ka}^\top R_kd_{z,k} \right), \qquad a\in[q].
\end{equation*}
Therefore, $H_{ab}^{\mathrm{ss}}$ collects the contributions of all dynamics equations that depend simultaneously on $z_a$ and $z_b$. In particular, the diagonal block $H_{aa}^{\mathrm{ss}}$ contains not only the contribution of $A_{aa}$ and $R_a$, but also the contribution of every block $A_{ka}$ through which $z_a$ influences the $k$th dynamics equation. Similarly, $H_{ab}^{\mathrm{ss}}$ may be nonzero even when $A_{ab}=A_{ba}=0$, provided that $z_a$ and $z_b$ influence at least one common dynamic block.

This motivates the \emph{steady-state sparsity graph}
\begin{equation*}
\mathcal G_{\mathrm{ss}} = \bigl( \{1,\dots,q\}, \mathcal E_{\mathrm{ss}} \bigr),
\end{equation*}
where
\begin{equation*}
\mathcal E_{\mathrm{ss}} := \left\{ \{a,b\}: a<b,\; H_{ab}^{\mathrm{ss}}\neq0 \right\}.
\end{equation*}
Since $H_{\mathrm{ss}}$ is symmetric, $H_{ba}^{\mathrm{ss}}=(H_{ab}^{\mathrm{ss}})^\top$, and it is sufficient to inspect one of the two corresponding blocks. The graph $\mathcal G_{\mathrm{ss}}$ therefore contains exactly the off-diagonal block couplings appearing in the reduced static objective.

Choose a family of \emph{steady-state clusters}
\begin{equation*}
\{\mathcal D_\alpha\}_{\alpha=1}^{s}, \qquad \mathcal D_\alpha\subseteq[q],
\end{equation*}
\green{such that every base block index belongs to at least one steady-state cluster and, for every edge $\{a,b\}\in\mathcal E_{\mathrm{ss}}$, there exists at least one $\alpha\in[s]$ such that $a,b\in\mathcal D_\alpha$.} For each edge $\{a,b\}\in\mathcal E_{\mathrm{ss}}$, define
\begin{equation*}
\mathcal I_{ab}^{\mathrm{ss}} := \left\{ \alpha\in[s]: a,b\in\mathcal D_\alpha \right\}.
\end{equation*}
By construction, $\mathcal I_{ab}^{\mathrm{ss}}\neq\emptyset$. Choose \emph{steady-state edge weights} $\rho_{ab,\alpha}\ge0$, for $\alpha\in\mathcal I_{ab}^{\mathrm{ss}}$, satisfying
\begin{equation}
\label{eq:ss_edge_partition_unity} \sum_{\alpha\in\mathcal I_{ab}^{\mathrm{ss}}} \rho_{ab,\alpha} = 1, \qquad \{a,b\}\in\mathcal E_{\mathrm{ss}},
\end{equation}
with $\rho_{ab,\alpha}=\rho_{ba,\alpha}$. Choose also \emph{steady-state diagonal weights} $\gamma_{a\alpha}>0$, for every $\alpha$ such that $a\in\mathcal D_\alpha$, satisfying
\begin{equation}
\label{eq:ss_gamma_partition_unity} \sum_{\alpha:\,a\in\mathcal D_\alpha} \gamma_{a\alpha} = 1, \qquad a\in[q].
\end{equation}
The weights $\gamma_{a\alpha}$ distribute the diagonal hessian blocks and the linear terms among the steady-state clusters containing block $a$, whereas the weights $\rho_{ab,\alpha}$ distribute each off-diagonal hessian block among the steady-state clusters containing the corresponding edge.

For a steady-state cluster
\begin{equation*}
\mathcal D_\alpha = \{a_1,\dots,a_{|\mathcal D_\alpha|}\},
\end{equation*}
define
\begin{equation*}
E_{\mathcal D_\alpha} = \begin{bmatrix} E_{a_1}&\cdots&E_{a_{|\mathcal D_\alpha|}} \end{bmatrix}, \qquad z_\alpha = E_{\mathcal D_\alpha}^\top z.
\end{equation*}
The local steady-state objective associated with $\mathcal D_\alpha$ is
\begin{equation*}
\Psi_\alpha(z_\alpha) = \frac{1}{2} z_\alpha^\top H_\alpha^{\mathrm{ss}} z_\alpha + \left( h_\alpha^{\mathrm{ss}} \right)^\top z_\alpha.
\end{equation*}
Its diagonal hessian blocks and linear terms are defined as
\begin{equation*}
\left( H_\alpha^{\mathrm{ss}} \right)_{aa} = \gamma_{a\alpha}H_{aa}^{\mathrm{ss}}, \qquad \left( h_\alpha^{\mathrm{ss}} \right)_a = \gamma_{a\alpha}h_a^{\mathrm{ss}}, \qquad a\in\mathcal D_\alpha.
\end{equation*}
For distinct $a,b\in\mathcal D_\alpha$, the off-diagonal blocks are defined by
\begin{equation*}
\left( H_\alpha^{\mathrm{ss}} \right)_{ab} = \begin{cases} \rho_{ab,\alpha}H_{ab}^{\mathrm{ss}}, & \{a,b\}\in\mathcal E_{\mathrm{ss}}, \quad \alpha\in\mathcal I_{ab}^{\mathrm{ss}}, \\
0, & \text{otherwise}. \end{cases}
\end{equation*}
Because of \eqref{eq:ss_gamma_partition_unity} and \eqref{eq:ss_edge_partition_unity}, the local contributions reconstruct the reduced steady-state objective exactly:
\begin{equation}
\label{eq:ss_cost_cluster_decomp} \Phi(z) = \sum_{\alpha=1}^{s} \Psi_\alpha \left( E_{\mathcal D_\alpha}^\top z \right) + \kappa_{\mathrm{ss}}.
\end{equation}
The constant $\kappa_{\mathrm{ss}}$ does not affect the optimizer and therefore does not need to be distributed among the local steady-state objectives. Consequently, solving the decomposed problem with consensus constraints on the base-block variables shared by different steady-state clusters is equivalent to solving the centralized static problem \eqref{eq:ss_static_problem_local_certificates}, and recovers the same unique optimizer $z^\star$.

Since $H_{\mathrm{ss}}\succ0$, the decomposed problem has the unique optimizer $z^\star$ of \eqref{eq:ss_static_problem_local_certificates}. The corresponding steady-state input and affine offset are
\begin{equation*}
s^\star=-\widetilde A z^\star-d_z, \qquad k_s^\star=d_z+(\widetilde A-K_s)z^\star.
\end{equation*}

\begin{remark}
\label{rem:graph_informed_cluster_selection}
The certificate cluster cover and the associated weights can be chosen from the block interaction structure of the dynamics. The partition $\Pi_{\mathcal B}$ is selected so that $R$, $\widetilde Q$, and $X$ are block diagonal, while $\mathcal G_\mathcal B$ identifies the couplings to be represented locally. An agent based cover yields small but potentially conservative LMIs, whereas neighborhoods, communities, or unions of strongly coupled blocks retain more of the centralized coupling structure at the cost of larger local constraints. Neutral splitting is obtained with
\begin{equation*}
\beta_{a\ell}=\frac{1}{|\{j\in[p]:a\in\mathcal C_j\}|},\qquad \theta_{ab,\ell}=\frac{1}{|\mathcal I_{ab}|}.
\end{equation*}
Alternatively, the weights may reflect interaction strength. They are fixed before solving the SDP to preserve convexity. Their optimization, together with optimal certificate cluster selection, is left for future work.
\end{remark}

\begin{remark}
\label{rem:ss_weights_exactness}
The steady-state weights $\gamma_{a\alpha}$ and $\rho_{ab,\alpha}$ only distribute the objective contributions among the local subproblems. Under \eqref{eq:ss_gamma_partition_unity}, \eqref{eq:ss_edge_partition_unity}, and consensus on shared variables, every admissible choice reconstructs the same objective $\Phi$ and therefore the same optimizer $z^\star$. Their choice may nevertheless affect numerical conditioning and the convergence of the distributed solver. In contrast, although the weights $\beta_{a\ell}$ and $\theta_{ab,\ell}$ also reconstruct the centralized LMI, imposing negativity of each local matrix is only sufficient for global negativity. Their choice can therefore affect feasibility and conservatism.
\end{remark}

\begin{remark}\label{rem:admm_pointer}
The local SDP certificates can be implemented through standard consensus-ADMM schemes; see, e.g., \citep{Boyd2011Distributed}. Each certificate cluster in the structured gain layer keeps local copies of the shared block variables, while the consensus step enforces agreement on duplicated variables. After convergence, the structured gain is assembled as $K_s=YX^{-1}$. The same principle applies to the steady-state layer, where each steady-state cluster keeps local copies of the variables $z_a$ it contains, and consensus constraints enforce agreement on shared blocks. Since the decomposition \eqref{eq:ss_cost_cluster_decomp} is exact, convergence of the consensus optimization recovers the unique optimizer $z^\star$ of the original static problem \eqref{eq:ss_static_problem_local_certificates}. The affine offset is then recovered as $k_s^\star=d_z+(\widetilde A-K_s)z^\star$. Since the contribution of the present paper is the local-to-global certificate rather than the distributed optimization algorithm itself, we do not further detail the ADMM iterations.
\end{remark}

\begin{figure}%[t]
    \centering
\includegraphics[width=0.8\linewidth,clip]{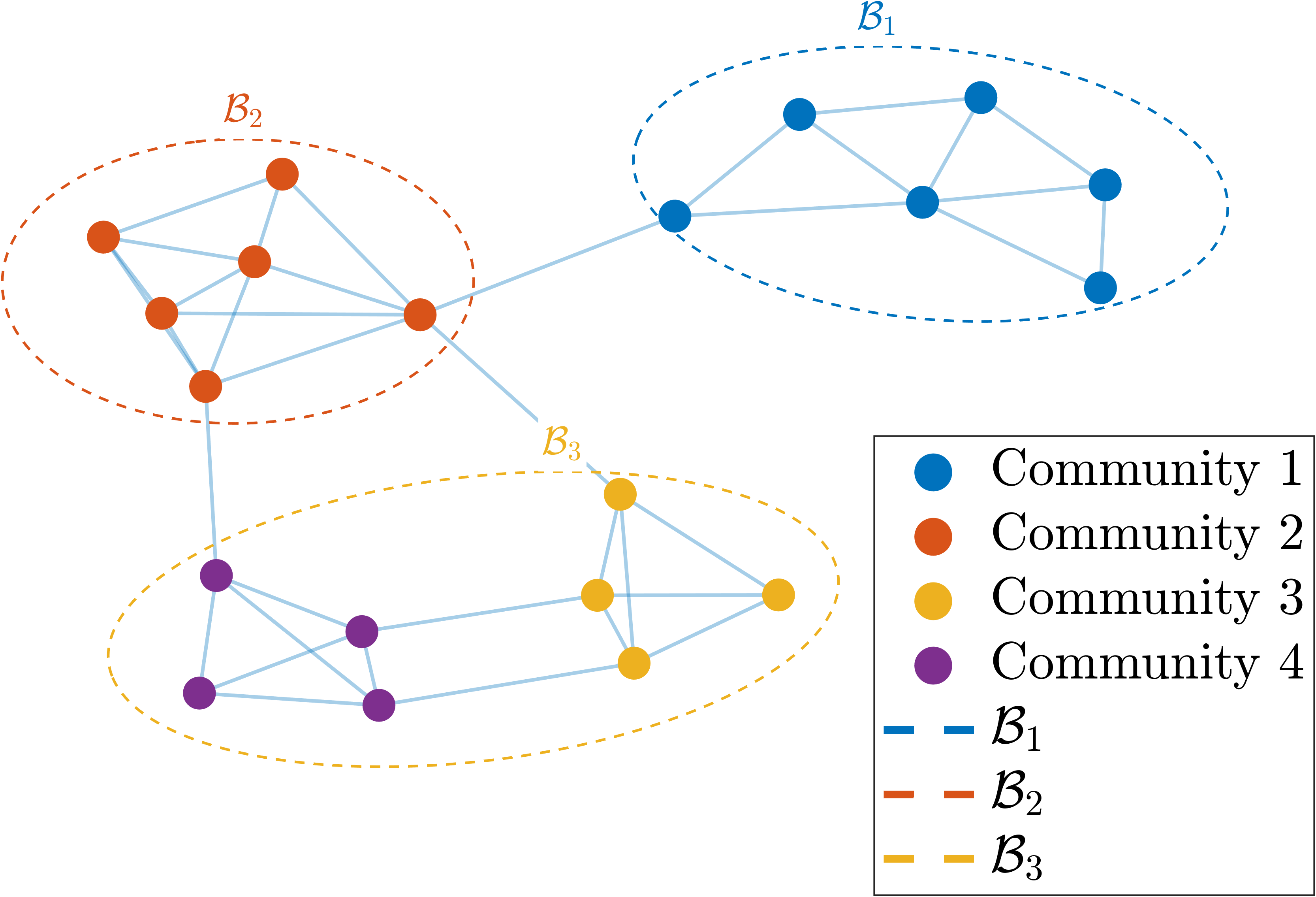}
    \vskip -5pt
\caption{Social interaction graph with four communities. The three base blocks $\mathcal B_1$, $\mathcal B_2$, and $\mathcal B_3$ coincide with the exposure groups. The structured gain layer uses the certificate clusters $\mathcal C_1=\{1,2\}$ and $\mathcal C_2=\{2,3\}$, whereas the steady-state layer uses the steady-state clusters $\mathcal D_1=\{1,2\}$, $\mathcal D_2=\{2,3\}$, and $\mathcal D_3=\{1,3\}$.}
    \label{fig:graph}
\end{figure} 

\section{Numerical Illustrations}
\label{sec:numerical-experiments}
This section illustrates the local-to-global two-layer design developed above on a network of agents with community structure. The example is organized to make each step of the construction explicit. We first specify the network, the model parameters, the base block partition, and the prescribed sparsity pattern of the feedback gain $K_s$. We then instantiate the local certificates for the structured gain layer and synthesize the structured gains for the dissipativity-based and $\mathcal H_2$-based formulations. For both designs, the local matrices are derived explicitly and their lifted sums are verified to reconstruct the corresponding centralized structured LMIs. \green{After the two structured gains have been obtained}, we construct the steady-state layer associated with the original normalized stage cost. We compute the exact reduced hessian $H_{\mathrm{ss}}$ and its steady-state sparsity graph, choose steady-state clusters that cover all resulting couplings, and decompose the static objective into local contributions. Since this decomposition is exact, the associated consensus problem recovers the unique optimizer $z^\star$ of the centralized steady-state problem. This optimal steady state is then used to recover the affine offsets associated with the two synthesized gains. Finally, we simulate the resulting closed-loop systems and compare the uncontrolled dynamics, the centralized Riccati-based controller, the structured dissipativity-based controller, and the structured $\mathcal H_2$-based controller.

\subsection{Network, model parameters, and certificate structure}
\label{subsec:numerical_setup}
\green{The network consists of $n=20$ agents divided into four communities of sizes $6$, $6$, $4$, and $4$. Each agent holds opinions on $m=3$ topics, so that the overall state is $z\in\reals^{60}$.} 
The inter-topic coupling matrix is 
\begin{equation*}
C= \begin{bmatrix} 0.7 & 0.2 & 0.1 \\
0.2 & 0.6 & 0.2 \\
0.1 & 0.2 & 0.7 \end{bmatrix}.
\end{equation*}
Matrix $C$ is symmetric, nonnegative, and row stochastic. Its eigenvalues are $1$, $0.6$, and $0.4$, and hence $C-I_3$ has one semisimple zero eigenvalue and two strictly stable eigenvalues, consistently with Assumption~\ref{ass:C}. The anchoring matrix is $A_a = \diag \left( 0.30\,\mathbbm{1}_6, 0.35\,\mathbbm{1}_6, 0.25\,\mathbbm{1}_8 \right).$
% \begin{equation*}
% A_a = \diag \left( 0.30\,\mathbbm{1}_6, 0.35\,\mathbbm{1}_6, 0.25\,\mathbbm{1}_8 \right).
% \end{equation*}
and the 
inner beliefs are constant within each base block and are chosen as $z_{\circ,i}=\col(0.65,-0.20,0.30)$ for $i\in\mathcal B_1$, $z_{\circ,i}=\col(0.05,0.45,-0.25)$ for $i\in\mathcal B_2$, and $z_{\circ,i}=\col(-0.60,-0.10,0.50)$ for $i\in\mathcal B_3$.
% The inner beliefs are constant inside each group and are defined by
% \begin{equation*}
% z_{\circ,i} = \begin{cases} \col(0.65,-0.20,0.30), & i\in\mathcal B_1, \\
% \col(0.05,0.45,-0.25), & i\in\mathcal B_2, \\
% \col(-0.60,-0.10,0.50), & i\in\mathcal B_3. \end{cases}
% \end{equation*}
To obtain a heterogeneous but reproducible initial condition, we set for each $i\in[20]$,  $\theta_i:=2\pi(i-1)/20$ and choose $z_i(0)=z_{\circ,i}+0.10\,\col\left(\sin(\theta_i),\cos(\theta_i),\sin(2\theta_i)\right)$.
% \begin{equation*}
% z_i(0) = z_{\circ,i} + 0.10 \begin{bmatrix} \sin\left(\dfrac{2\pi(i-1)}{20}\right) \\[1mm]
% \cos\left(\dfrac{2\pi(i-1)}{20}\right) \\[1mm]
% \sin\left(\dfrac{4\pi(i-1)}{20}\right) \end{bmatrix}, \qquad i\in[20].
% \end{equation*}

%The exposure partition contains three groups because the two communities of $4$ agents are combined into a single exposure group.

For this example, we decide to divide the agents into the three base blocks $\mathcal B_1=\{1,\ldots,6\}$, $\mathcal B_2=\{7,\ldots,12\}$, and $\mathcal B_3=\{13,\ldots,20\}$.
% \begin{equation*}
% \mathcal B_1=\{1,\ldots,6\}, \quad \mathcal B_2=\{7,\ldots,12\}, \quad \mathcal B_3=\{13,\ldots,20\}.
% \end{equation*}
The graph contains within-base-block connections and between-base-block connections only between $\mathcal B_1$ and $\mathcal B_2$, and between $\mathcal B_2$ and $\mathcal B_3$. With respect to the partition $\Pi_{\mathcal B}=\{\mathcal B_1,\mathcal B_2,\mathcal B_3\}$,
% \begin{equation*}
% \Pi_{\mathcal B}=\{\mathcal B_1,\mathcal B_2,\mathcal B_3\},
% \end
\green{the base-block interaction graph $\mathcal G_\mathcal B$ is a path graph, and the block dimensions corresponding to each base block are respectively $d_1=d_2=18$ and $d_3=24.$}
The exposure partition used in the performance index $J\sbs{F}$ coincides with the base-block partition, namely $\Pi_{\mathcal F}=\Pi_{\mathcal B}$, so that $L_u$ contains only within-base-block recommendation couplings. The performance weights are fixed as $W_{\mathrm{EN}}=0.6I_{60}$, $W_{\mathrm P}=0.4I_{60}$, $W_{\mathrm D}=I_{60}$, $W_{\mathrm{EX}}=1.5I_{60}$, and $W_{\mathrm C}=0.8I_{60}$, with $\alpha_{\mathrm F}=0.25$.
% \begin{equation*}
% W_{\mathrm{EN}}=0.6I_{60}, \qquad W_{\mathrm P}=0.4I_{60}, \qquad W_{\mathrm D}=I_{60},
% \end{equation*}
% \begin{equation*}
% W_{\mathrm{EX}}=1.5I_{60}, \qquad W_{\mathrm C}=0.8I_{60}, \qquad \alpha_{\mathrm F}=0.25.
% \end{equation*}
Consequently, the selected weights satisfy Lemma~3 in \citep{mariano_frasca_companion}, and thus $\widetilde Q \succ 0$, which implies that they satisfy Assumption~\ref{ass:strict_pd_reduced_cost}. Moreover, since $R^{-1}N$ is block diagonal with respect to $\Pi_{\mathcal B}$, the off-diagonal block sparsity pattern of $\widetilde A=A_{cz}-R^{-1}N$ coincides with that of $A_{cz}$. 

We restrict the certificate matrix and the feedback change-of-variable matrix to
\begin{equation*}
X=\diag(X_1,X_2,X_3), \qquad Y= \begin{bmatrix} Y_{11} & Y_{12} & 0 \\
Y_{21} & Y_{22} & Y_{23} \\
0 & Y_{32} & Y_{33} \end{bmatrix},
\end{equation*}
where $X_1,X_2\in\mathbb S_{\succ0}^{18}$, $X_3\in\mathbb S_{\succ0}^{24}$. Since $X^{-1}$ is block diagonal, the recovered gain $K_s=YX^{-1}$ has the same block sparsity pattern as $Y$.

For both structured gain synthesis formulations, we use the two overlapping certificate clusters $\mathcal C_1=\{1,2\}$ and $\mathcal C_2=\{2,3\}$.
% \begin{equation*}
% \mathcal C_1=\{1,2\}, \qquad \mathcal C_2=\{2,3\}.
% \end{equation*}
Each off-diagonal interaction belongs to a unique certificate cluster, and hence the edge weights are $\theta_{12,1}=1$, and $\theta_{23,2}=1$, while 
% \begin{equation*}
% \theta_{12,1}=1, \qquad \theta_{23,2}=1.
% \end{equation*}
the diagonal block contributions are split according to $\beta_{1,1}=1$, $\beta_{2,1}=\beta_{2,2}=\frac{1}{2}$, and $\beta_{3,2}=1$.
% \begin{equation*}
% \beta_{1,1}=1, \qquad \beta_{2,1}=\beta_{2,2}=\frac{1}{2}, \qquad \beta_{3,2}=1.
% \end{equation*}
These weights satisfy the required partition-of-unity conditions and will be used in both the dissipativity-based and the $\mathcal H_2$-based local certificates.

\subsection{Dissipativity-based local synthesis}
\label{subsec:numerical_dissipativity}
We first instantiate the local certificates associated with the dissipativity-based formulation. The block sparsity pattern, certificate clusters, and weights are those defined in Section~\ref{subsec:numerical_setup}. Since the base-block interaction graph is a path, the off-diagonal interactions $\{1,2\}$ and $\{2,3\}$ are assigned to $\mathcal C_1$ and $\mathcal C_2$, respectively, while the diagonal contribution of block $2$ is equally divided between the two overlapping certificate clusters.

For the first certificate cluster $\mathcal C_1=\{1,2\}$, define
\begin{equation*}
A_{\mathcal C_1} = \begin{bmatrix} A_{11} & A_{12} \\
A_{21} & \frac{1}{2}A_{22} \end{bmatrix}, \qquad Y_{\mathcal C_1} = \begin{bmatrix} Y_{11} & Y_{12} \\
Y_{21} & \frac{1}{2}Y_{22} \end{bmatrix},
\end{equation*}
whereas, for $\mathcal C_2=\{2,3\}$, define
\begin{equation*}
A_{\mathcal C_2} = \begin{bmatrix} \frac{1}{2}A_{22} & A_{23} \\
A_{32} & A_{33} \end{bmatrix}, \qquad Y_{\mathcal C_2} = \begin{bmatrix} \frac{1}{2}Y_{22} & Y_{23} \\
Y_{32} & Y_{33} \end{bmatrix}.
\end{equation*}

The associated unweighted and weighted matrices are $X_{\mathcal C_1}^{\mathrm A}=\diag(X_1,X_2)$, $X_{\mathcal C_1}^{\beta}=\diag\left(X_1,\frac{1}{2}X_2\right)$, $X_{\mathcal C_2}^{\mathrm A}=\diag(X_2,X_3)$, and $X_{\mathcal C_2}^{\beta}=\diag\left(\frac{1}{2}X_2,X_3\right)$. Similarly, $Q_{\mathcal C_1}^{-1}=\diag\left(Q_1^{-1},\frac{1}{2}Q_2^{-1}\right)$, $R_{\mathcal C_1}^{-1}=\diag\left(R_1^{-1},\frac{1}{2}R_2^{-1}\right)$, $Q_{\mathcal C_2}^{-1}=\diag\left(\frac{1}{2}Q_2^{-1},Q_3^{-1}\right)$, and $R_{\mathcal C_2}^{-1}=\diag\left(\frac{1}{2}R_2^{-1},R_3^{-1}\right)$.
% The associated unweighted and weighted certificate matrices are
% \begin{equation*}
% X_{\mathcal C_1}^{\mathrm A} = \diag(X_1,X_2), \qquad X_{\mathcal C_1}^{\beta} = \diag\left(X_1,\frac{1}{2}X_2\right),
% \end{equation*}
% and
% \begin{equation*}
% X_{\mathcal C_2}^{\mathrm A} = \diag(X_2,X_3), \quad X_{\mathcal C_2}^{\beta} = \diag\left(\frac{1}{2}X_2,X_3\right).
% \end{equation*}
% Similarly,
% \begin{equation*}
% Q_{\mathcal C_1}^{-1} = \diag\left(Q_1^{-1},\frac{1}{2}Q_2^{-1}\right), \quad R_{\mathcal C_1}^{-1} = \diag\left(R_1^{-1},\frac{1}{2}R_2^{-1}\right),
% \end{equation*}
% and
% \begin{equation*}
% Q_{\mathcal C_2}^{-1} = \diag\left(\frac{1}{2}Q_2^{-1},Q_3^{-1}\right), \quad R_{\mathcal C_2}^{-1} = \diag\left(\frac{1}{2}R_2^{-1},R_3^{-1}\right).
% \end{equation*}

The corresponding local certificate matrices are
\begin{equation*}
 \mathcal L_{\mathcal C_\ell} = \begin{bmatrix} \Lambda_{\mathcal C_\ell} & X_{\mathcal C_\ell}^{\beta} & Y_{\mathcal C_\ell}^{\top} \\
X_{\mathcal C_\ell}^{\beta} & -Q_{\mathcal C_\ell}^{-1} & 0 \\
Y_{\mathcal C_\ell} & 0 & -R_{\mathcal C_\ell}^{-1} \end{bmatrix}, \qquad \ell\in\{1,2\},
\end{equation*}
where
\begin{equation*}
\Lambda_{\mathcal C_\ell} = \Sym\left( A_{\mathcal C_\ell}X_{\mathcal C_\ell}^{\mathrm A} - Y_{\mathcal C_\ell} \right).
\end{equation*}

The structured gain is obtained by solving
\begin{equation*}
 \begin{aligned} \min_{\{X_a\},Y,\{Z_a\}} \quad & \sum_{a=1}^{3}\Tr(Z_a) \\
\mathrm{s.t.}\quad & X_a\succ0, \qquad a\in\{1,2,3\}, \\
& \mathcal L_{\mathcal C_\ell}\prec0, \qquad \ell\in\{1,2\}, \\
& \begin{bmatrix} Z_a&I \\
I&X_a \end{bmatrix} \succeq0, \qquad a\in\{1,2,3\}, \\
& Y_{13}=0, \qquad Y_{31}=0. \end{aligned}
\end{equation*}
Denoting the solution by $X^{\mathrm{diss}}$ and $Y^{\mathrm{diss}}$, the structured gain is recovered as
\begin{equation*}
 K_s^{\mathrm{diss}} = Y^{\mathrm{diss}} \left(X^{\mathrm{diss}}\right)^{-1}.
\end{equation*}
The corresponding affine offset is computed after determining the common optimal steady state $(z^\star,s^\star)$.

\subsection{Local $\mathcal H_2$ synthesis}
\label{subsec:numerical_h2}

We next synthesize a second structured gain using the local $\mathcal H_2$ formulation. We choose $B_0=I_{60}$ and $W_0:=B_0B_0^\top=I_{60}$ so that the transient response is evaluated uniformly over all opinion-state directions. Partitioning $W_0=\diag(W_{0,1},W_{0,2},W_{0,3})$, one has $W_{0,1}=W_{0,2}=I_{18}$ and $W_{0,3}=I_{24}$.

For the certificate clusters $\mathcal C_1=\{1,2\}$ and $\mathcal C_2=\{2,3\}$, define $W_{0,\mathcal C_1}^{\beta}=\diag(I_{18},\frac{1}{2}I_{18})$ and $W_{0,\mathcal C_2}^{\beta}=\diag(\frac{1}{2}I_{18},I_{24})$, respectively. Using the local matrices introduced in the previous subsection, the local Lyapunov inequalities are $\mathcal N_{\mathcal C_\ell}=\Sym(A_{\mathcal C_\ell}X_{\mathcal C_\ell}^{\mathrm A}-Y_{\mathcal C_\ell})+W_{0,\mathcal C_\ell}^{\beta}$, for $\ell\in\{1,2\}$.
For the first certificate cluster, define $C_{1,\mathcal C_1}^{\beta}=\col(\diag(Q_1^{1/2},\frac{1}{2}Q_2^{1/2}),0)$ and $D_{12,\mathcal C_1}=\col(0,\diag(R_1^{1/2},R_2^{1/2}))$, while for the second certificate cluster define $C_{1,\mathcal C_2}^{\beta}=\col(\diag(\frac{1}{2}Q_2^{1/2},Q_3^{1/2}),0)$ and $D_{12,\mathcal C_2}=\col(0,\diag(R_2^{1/2},R_3^{1/2}))$. The corresponding local output matrices are $M_{\mathcal C_\ell}=C_{1,\mathcal C_\ell}^{\beta}X_{\mathcal C_\ell}^{\mathrm A}-D_{12,\mathcal C_\ell}Y_{\mathcal C_\ell}$, for $\ell\in\{1,2\}$.

Finally, introduce the local epigraph variables $Z_{\mathcal C_1}^{\beta}=\diag(Z_1^{\mathrm Q},\frac{1}{2}Z_2^{\mathrm Q},Z_1^{\mathrm R},\frac{1}{2}Z_2^{\mathrm R})$ and $Z_{\mathcal C_2}^{\beta}=\diag(\tfrac{1}{2}Z_2^{\mathrm Q},\allowbreak Z_3^{\mathrm Q},\allowbreak \tfrac{1}{2}Z_2^{\mathrm R},\allowbreak Z_3^{\mathrm R})$.
The local output LMIs are
\begin{equation*}
 \mathcal O_{\mathcal C_\ell} = \begin{bmatrix} Z_{\mathcal C_\ell}^{\beta} & M_{\mathcal C_\ell} \\
M_{\mathcal C_\ell}^{\top} & X_{\mathcal C_\ell}^{\beta} \end{bmatrix}, \qquad \ell\in\{1,2\}.
\end{equation*}

The local $\mathcal H_2$ gain is obtained by solving
\begin{equation*}
 \begin{aligned} \min_{\{X_a\},Y,\{Z_a^{\mathrm Q}\},\{Z_a^{\mathrm R}\}} \quad & \sum_{a=1}^{3}\Tr\left(Z_a^{\mathrm Q}\right) + \sum_{a=1}^{3}\Tr\left(Z_a^{\mathrm R}\right) \\
\mathrm{s.t.}\quad & X_a\succ0, \qquad a\in\{1,2,3\}, \\
& \mathcal N_{\mathcal C_\ell}\prec0, \qquad \ell\in\{1,2\}, \\
& \mathcal O_{\mathcal C_\ell}\succeq0, \qquad \ell\in\{1,2\}, \\
& Y_{13}=0, \qquad Y_{31}=0. \end{aligned}
\end{equation*}
Denoting its solution by $X^{\mathcal H_2}$ and $Y^{\mathcal H_2}$, the structured gain is recovered as
\begin{equation*}
 K_s^{\mathcal H_2} = Y^{\mathcal H_2} \left(X^{\mathcal H_2}\right)^{-1}.
\end{equation*}
As in the dissipativity-based design, the corresponding affine offset is computed after determining the common optimal steady state $(z^\star,s^\star)$.

\subsection{Local computation of the optimal steady state}
\label{subsec:numerical_steady_state}

We finally compute the optimal steady state associated with the original normalized stage cost. The reduced static objective is
\begin{equation*}
\Phi(z) = \frac{1}{2}z^\top H_{\mathrm{ss}}z + h_{\mathrm{ss}}^\top z + \kappa_{\mathrm{ss}},
\end{equation*}
where
\begin{equation*}
H_{\mathrm{ss}} = 2\left( \widetilde Q + \widetilde A^\top R\widetilde A \right), \qquad h_{\mathrm{ss}} = 2\left( c+\widetilde A^\top Rd_z \right).
\end{equation*}
Partition these quantities according to $\Pi_{\mathcal B}$ as
\begin{equation*}
H_{\mathrm{ss}} = \begin{bmatrix} H_{11}^{\mathrm{ss}} & H_{12}^{\mathrm{ss}} & H_{13}^{\mathrm{ss}} \\
H_{21}^{\mathrm{ss}} & H_{22}^{\mathrm{ss}} & H_{23}^{\mathrm{ss}} \\
H_{31}^{\mathrm{ss}} & H_{32}^{\mathrm{ss}} & H_{33}^{\mathrm{ss}} \end{bmatrix}, \qquad h_{\mathrm{ss}} = \col\left( h_1^{\mathrm{ss}}, h_2^{\mathrm{ss}}, h_3^{\mathrm{ss}} \right),
\end{equation*}
and $\kappa_{\mathrm{ss}} = d_z^\top Rd_z$. Although the base-block interaction graph is a path, the steady-state hessian generally contains the additional two-hop coupling
\begin{equation*}
H_{13}^{\mathrm{ss}} = 2A_{21}^\top R_2A_{23}.
\end{equation*}
We therefore use the steady-state clusters $\mathcal D_1=\{1,2\}$, $\mathcal D_2=\{2,3\}$, and $\mathcal D_3=\{1,3\}.$
% \begin{equation*}
% \mathcal D_1=\{1,2\}, \qquad \mathcal D_2=\{2,3\}, \qquad \mathcal D_3=\{1,3\}.
% \end{equation*}
Each off-diagonal block is assigned to the unique steady-state cluster containing the corresponding pair, while each diagonal and linear contribution is equally divided between the two steady-state clusters containing that block.

The local hessian and linear terms are consequently
\begin{equation*}
H_{\mathcal D_1}^{\mathrm{ss}} = \begin{bmatrix} \frac{1}{2}H_{11}^{\mathrm{ss}} & H_{12}^{\mathrm{ss}} \\
H_{21}^{\mathrm{ss}} & \frac{1}{2}H_{22}^{\mathrm{ss}} \end{bmatrix}, \qquad h_{\mathcal D_1}^{\mathrm{ss}} = \frac{1}{2} \col\left( h_1^{\mathrm{ss}}, h_2^{\mathrm{ss}} \right),
\end{equation*}
\begin{equation*}
H_{\mathcal D_2}^{\mathrm{ss}} = \begin{bmatrix} \frac{1}{2}H_{22}^{\mathrm{ss}} & H_{23}^{\mathrm{ss}} \\
H_{32}^{\mathrm{ss}} & \frac{1}{2}H_{33}^{\mathrm{ss}} \end{bmatrix}, \qquad h_{\mathcal D_2}^{\mathrm{ss}} = \frac{1}{2} \col\left( h_2^{\mathrm{ss}}, h_3^{\mathrm{ss}} \right),
\end{equation*}
and
\begin{equation*}
H_{\mathcal D_3}^{\mathrm{ss}} = \begin{bmatrix} \frac{1}{2}H_{11}^{\mathrm{ss}} & H_{13}^{\mathrm{ss}} \\
H_{31}^{\mathrm{ss}} & \frac{1}{2}H_{33}^{\mathrm{ss}} \end{bmatrix}, \qquad h_{\mathcal D_3}^{\mathrm{ss}} = \frac{1}{2} \col\left( h_1^{\mathrm{ss}}, h_3^{\mathrm{ss}} \right).
\end{equation*}

For $\alpha\in\{1,2,3\}$, define
\begin{equation*}
\Psi_{\mathcal D_\alpha}(z_{\mathcal D_\alpha}) = \frac{1}{2} z_{\mathcal D_\alpha}^\top H_{\mathcal D_\alpha}^{\mathrm{ss}} z_{\mathcal D_\alpha} + \left( h_{\mathcal D_\alpha}^{\mathrm{ss}} \right)^\top z_{\mathcal D_\alpha}.
\end{equation*}
The centralized objective is then reconstructed exactly as
\begin{align*}
\Phi(z) = \Psi_{\mathcal D_1}\left(\col(z_1,z_2)\right) &+ \Psi_{\mathcal D_2}\left(\col(z_2,z_3)\right) \\
&+ \Psi_{\mathcal D_3}\left(\col(z_1,z_3)\right) + \kappa_{\mathrm{ss}}.
\end{align*}
Solving the three local quadratic problems with consensus imposed on the duplicated block variables therefore recovers the unique centralized optimizer $z^\star$. The corresponding steady-state input is
\begin{equation*}
s^\star = -\widetilde A z^\star-d_z.
\end{equation*}
Finally, the affine offsets associated with the two structured gains are
\begin{equation*}
k_s^{\mathrm{diss}} = d_z+ \left( \widetilde A-K_s^{\mathrm{diss}} \right)z^\star, \quad
k_s^{\mathcal H_2} = d_z+ \left( \widetilde A-K_s^{\mathcal H_2} \right)z^\star.
\end{equation*}
Both structured controllers therefore converge to the same optimal steady-state pair $(z^\star,s^\star)$, while differing in their transient gains.

\subsection{Closed-loop simulations and performance comparison}
\label{subsec:numerical_results}

We now compare the closed-loop behavior of the synthesized controllers. The simulations include the uncontrolled opinion dynamics, the centralized Riccati controller, the structured dissipativity-based controller, and the structured $\mathcal H_2$-based controller. The two structured synthesis problems are solved via the standard consensus-ADMM scheme in \citep{Boyd2011Distributed}. 

The corresponding opinion trajectories are reported in Fig.~\ref{fig:traj}. Thin curves represent the trajectories of the individual agents, while the thicker curves show the average opinion within each community. The uncontrolled trajectories converge according to the combined effect of social interaction, topic coupling, and anchoring. The controlled trajectories converge to the common optimal steady state, although their transient responses differ because of the different feedback gains. Figure~\ref{fig:opinion-range} shows the distribution of the opinions within each community at the final simulation time, highlighting how the controllers affect the residual dispersion of the individual opinions within the same community.

\begin{figure}[t]
    \centering
    \includegraphics[width=\linewidth,clip]{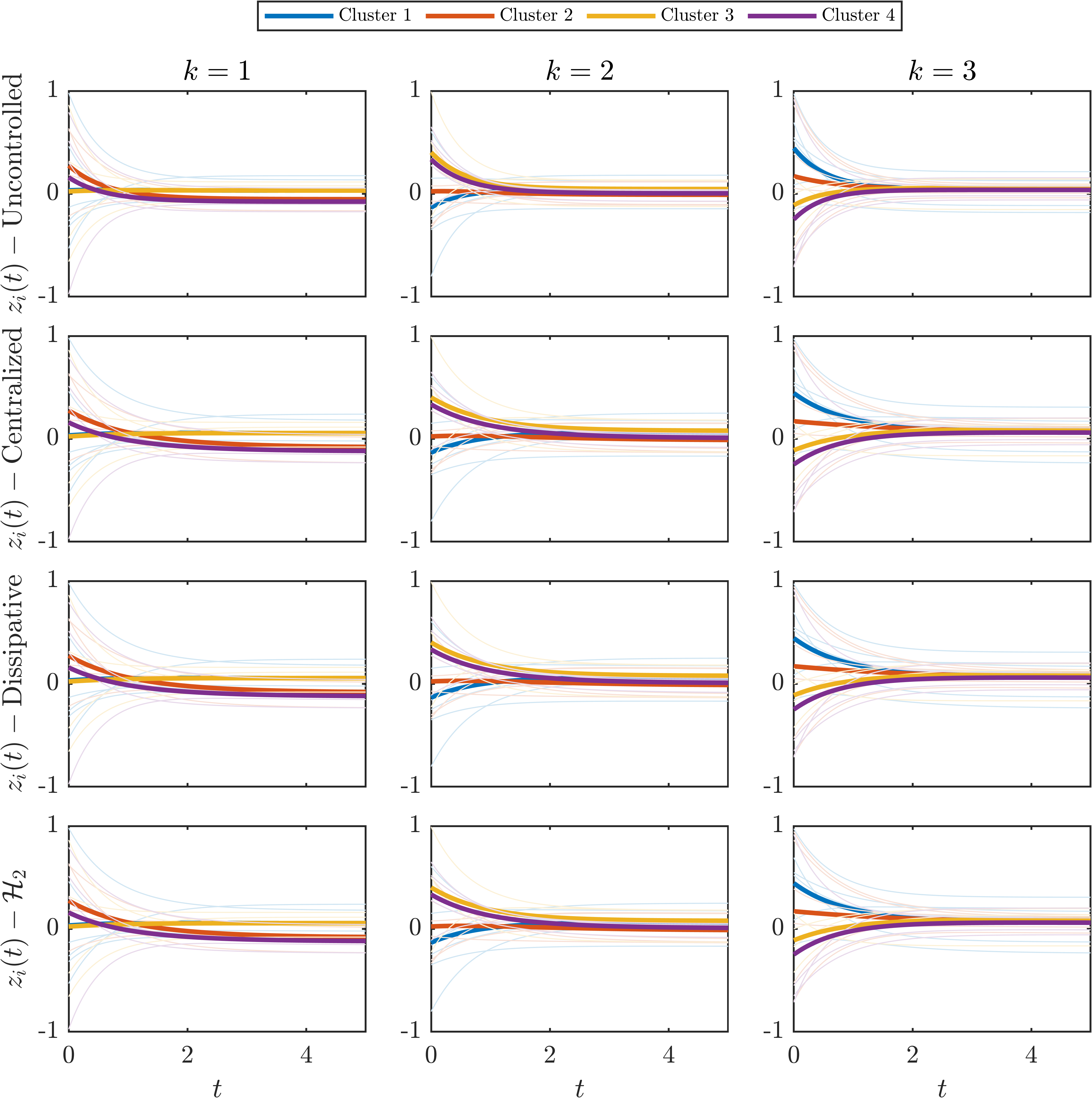}
    \vskip -5pt
    \caption{Opinion trajectories under the uncontrolled dynamics, the centralized Riccati controller, the structured dissipativity-based controller, and the structured $\mathcal H_2$ controller for the topics $k\in[3]$. Thin lines represent individual agents, while thick lines represent community averages.}
    \label{fig:traj}
\end{figure}

\begin{figure}[t]
    \centering
    \includegraphics[width=\linewidth,clip]{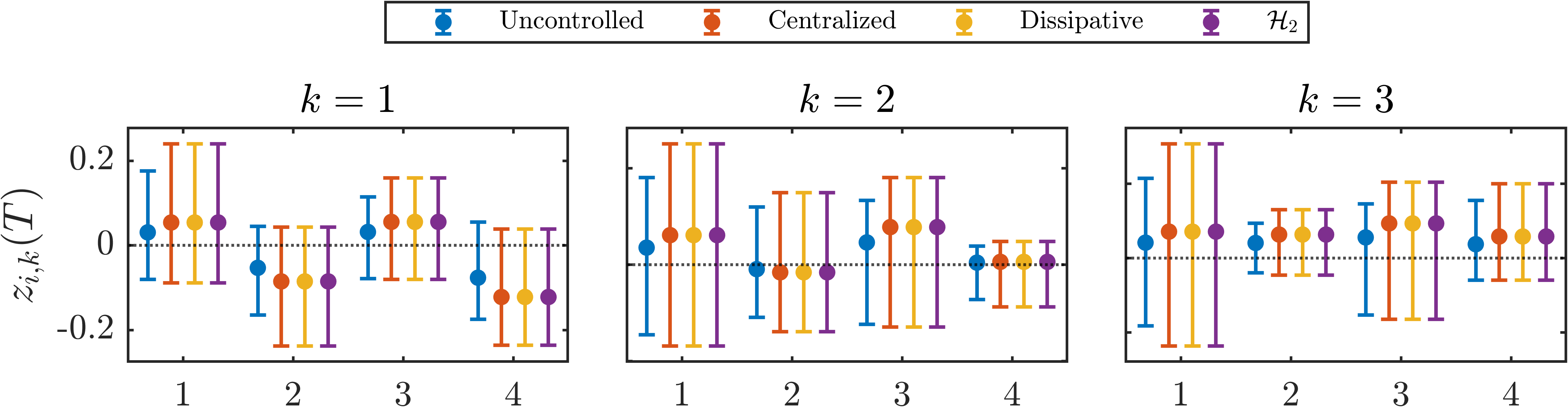}
    \vskip -5pt
    \caption{Range of the opinions within each community at the final simulation time.}
    \label{fig:opinion-range}
\end{figure}

The transient performance is compared in Fig.~\ref{fig:cost-ratio}. The figure reports the cumulative value of the quadratic density $\tilde z^\top\widetilde Q\tilde z+\tilde s^\top R\tilde s$, normalized by the corresponding value obtained with the centralized Riccati controller. Since the plotted quantities are finite-horizon cumulative costs, the ratio is not required to remain above one at every intermediate time. The Riccati controller minimizes the complete infinite-horizon cost, rather than any partial finite-horizon integral. The final part of the simulation therefore provides the more meaningful comparison once the transients have sufficiently decayed.

\begin{figure}[t]
    \centering
    \includegraphics[width=\linewidth,clip]{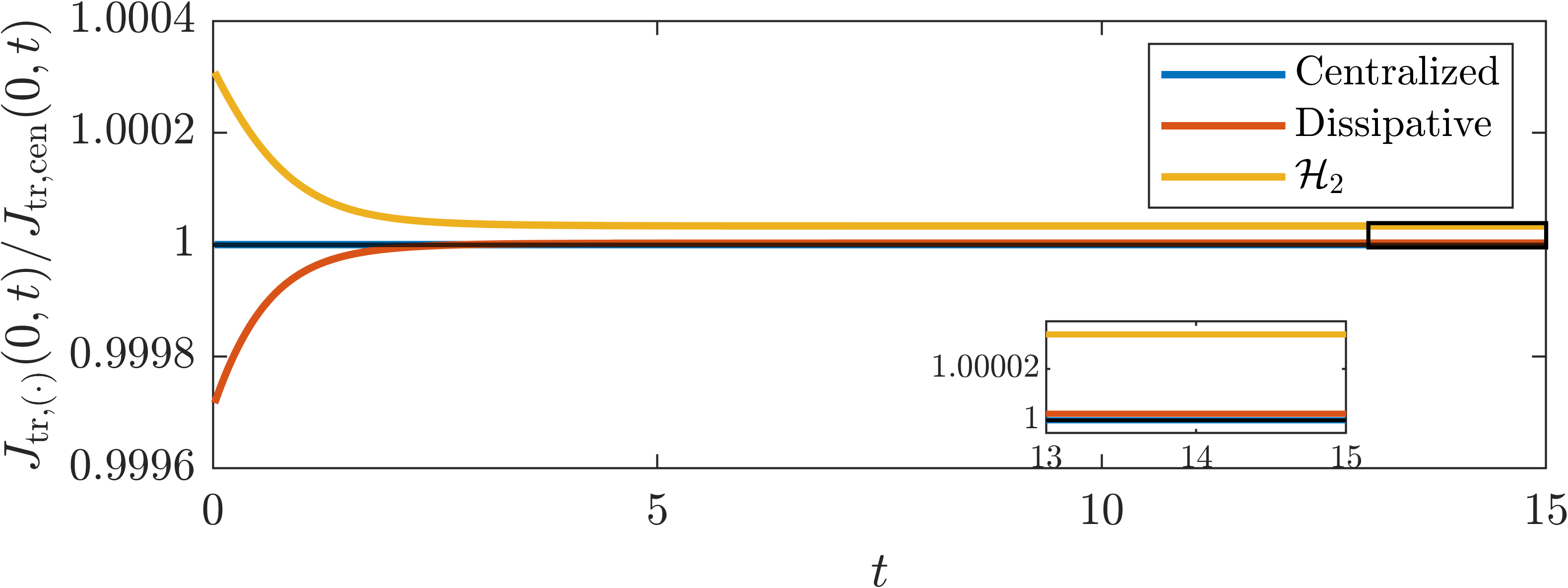}
    \vskip -5pt
    \caption{Finite-horizon cumulative quadratic energy normalized by the corresponding value obtained with the centralized Riccati controller.}
    \label{fig:cost-ratio}
\end{figure}

Figure~\ref{fig:cost-components} separates the cumulative physical performance index into its individual contributions. These include the engagement reward, polarization penalty, deviation penalty, recommendation effort, graph-regularization term, and recommendation mismatch. The comparison shows how the different feedback gains redistribute the overall performance among the individual terms. 

\begin{figure}[t]
    \centering
    \includegraphics[width=\linewidth,clip]{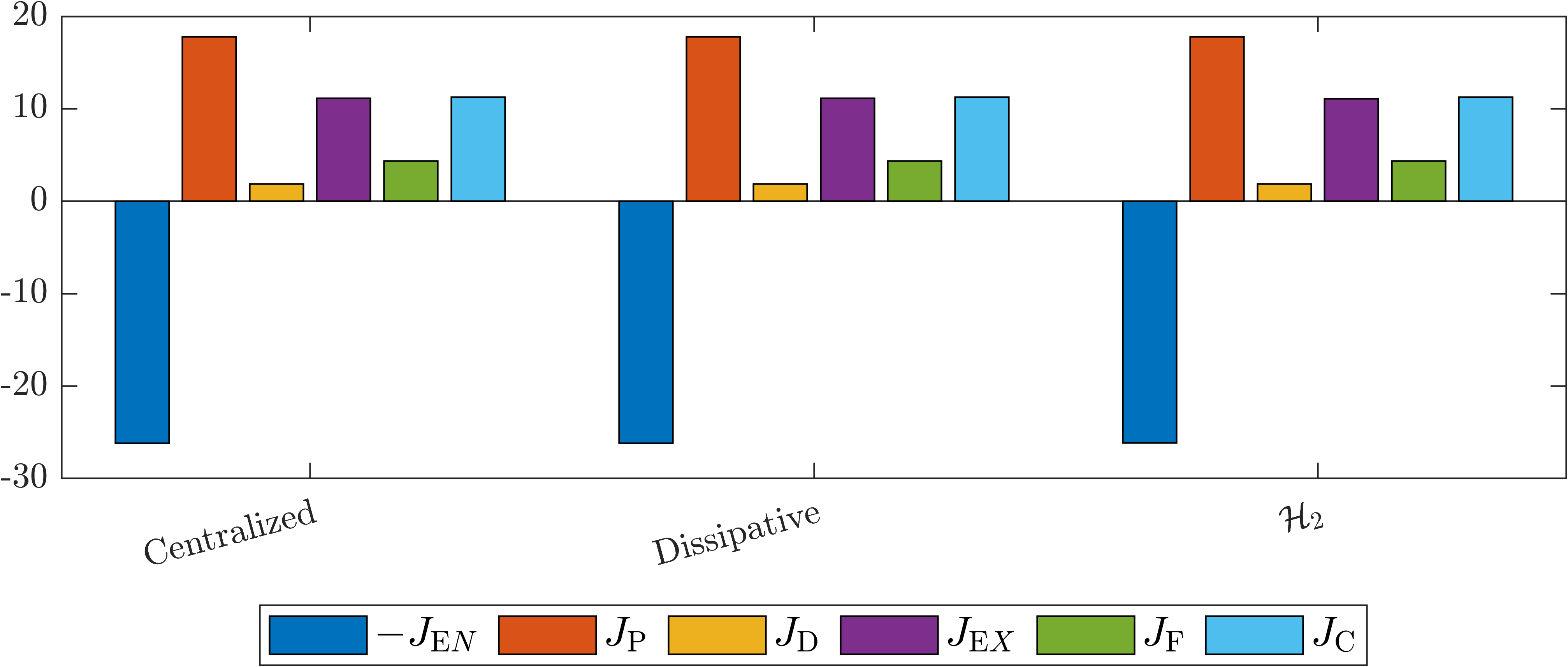}
    \vskip -5pt
    \caption{Cumulative contributions to the cost.}
    \label{fig:cost-components}
\end{figure}

The agent-wise sparsity patterns of the feedback gains are reported in Fig.~\ref{fig:sparsity-patterns}. A dot at position $(i,j)$ means that the recommendation applied to agent $i$ uses at least one opinion coordinate of agent $j$. Each displayed entry therefore summarizes the complete $m\times m$ topic-level block associated with the pair of agents. The centralized Riccati controller uses information from all the network, whereas the two structured controllers retain only the dependencies allowed by the prescribed information pattern. Their binary patterns are the same because the dissipativity and $\mathcal H_2$ designs use the same sparsity mask, although the numerical values of their nonzero gain blocks are generally different.

\begin{figure}[t]
    \centering
    \begin{subfigure}[t]{0.45\linewidth}
        \centering
        \includegraphics[width=\linewidth,clip]{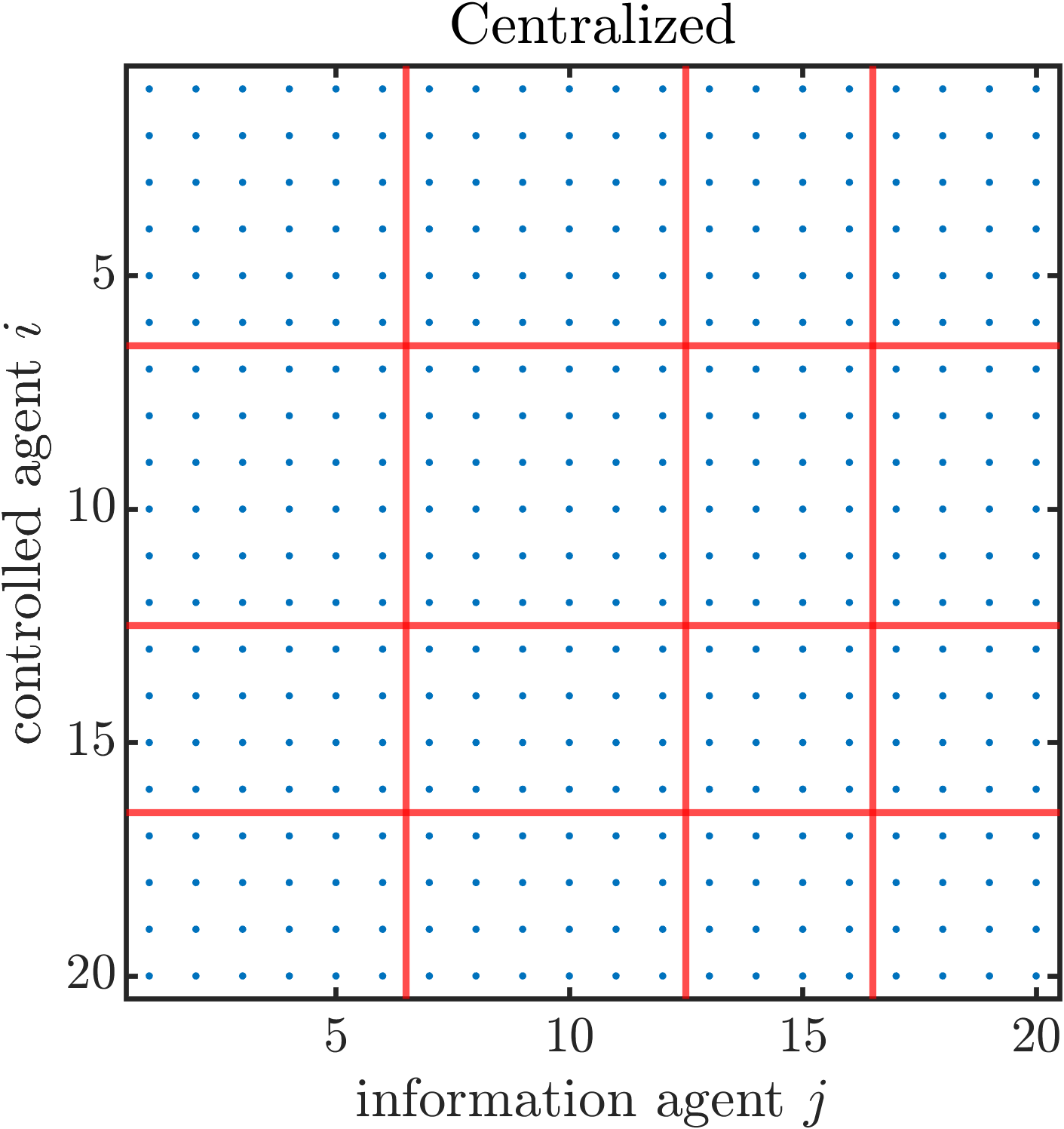}
        \caption{Centralized gain.}
        \label{fig:sparsity-centralized}
    \end{subfigure}
    \hfill
    \begin{subfigure}[t]{0.45\linewidth}
        \centering
        \includegraphics[width=\linewidth,clip]{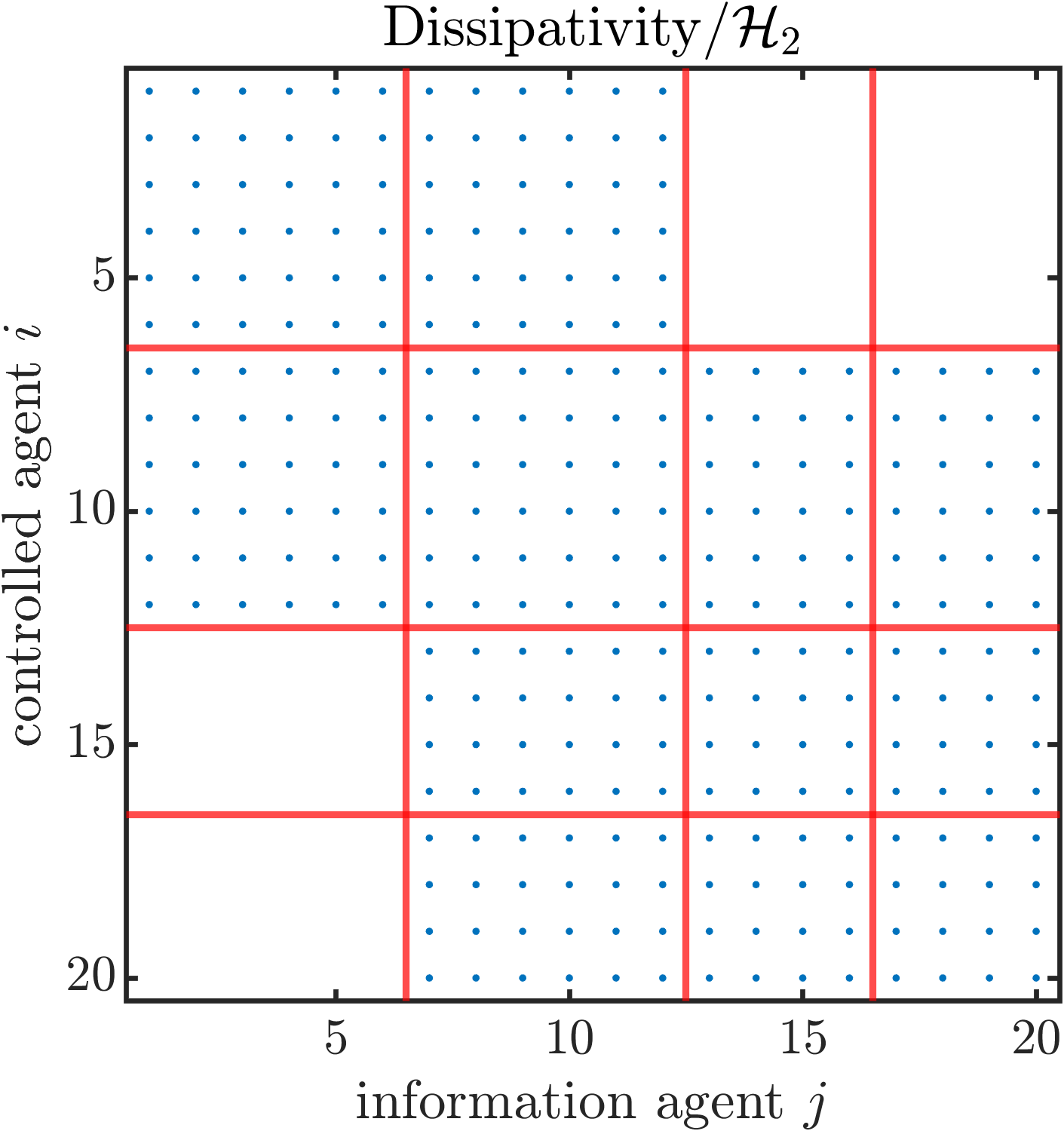}
        \caption{Dissipativity/$\mathcal H_2$ gains.}
        \label{fig:sparsity-structured}
    \end{subfigure}
    \vskip -5pt
    \caption{Agent-wise sparsity patterns of the synthesized centralized (left) and dissipativity/ $\mathcal H_2$-based (right) feedback gains. A dot at $(i,j)$ indicates that at least one opinion coordinate of agent $j$ is used in the recommendation applied to agent $i$.}
    \label{fig:sparsity-patterns}
     \vskip -3pt
\end{figure}

% \begin{figure}[t]
%     \centering
%     \includegraphics[width=\linewidth,clip]{figs/Fig_6.png}
%     \vskip -5pt
%     \caption{Agent-wise sparsity patterns of the synthesized feedback gains. A dot at $(i,j)$ indicates that at least one opinion coordinate of agent $j$ is used in the recommendation applied to agent $i$.}
%     \label{fig:sparsity-patterns}
% \end{figure}

Finally, Fig.~\ref{fig:polarization-sweep} illustrates the effect of varying the polarization weight $w_{\mathrm P}$. Increasing this weight modifies the balance between polarization reduction, engagement, deviation from the uncontrolled equilibrium, recommendation mismatch, recommendation effort, and graph regularization.

\begin{figure}[t]
    \centering
    \includegraphics[width=\linewidth,clip]{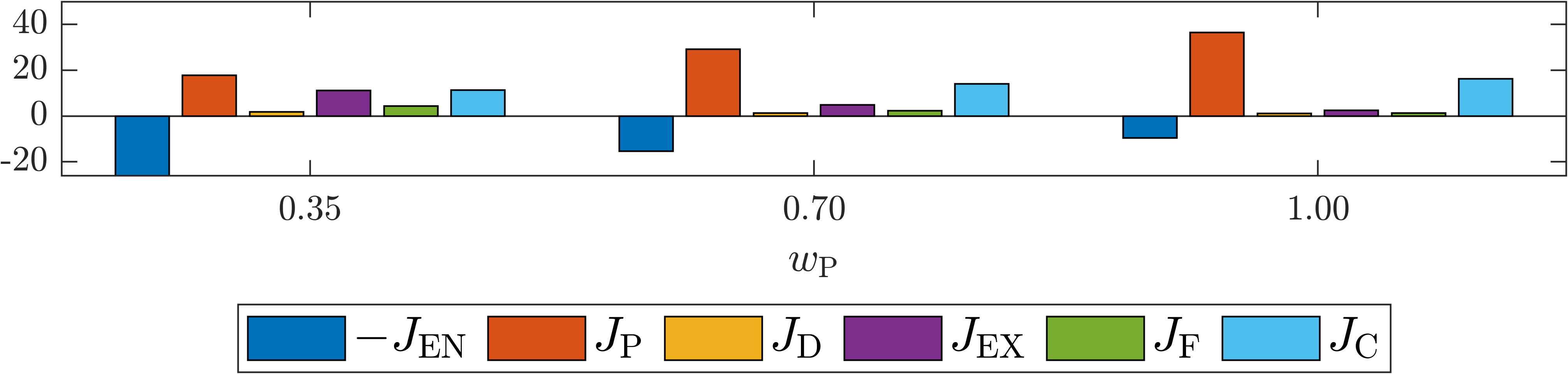}
    \vskip -5pt
    \caption{Cumulative contributions to the cost for different values of the polarization weight $w_{\mathrm P}$.}
    \label{fig:polarization-sweep}
       \vskip -3pt
\end{figure}
\section{Discussion and conclusion}
\label{sec:disc}

This paper considered the synthesis of closed-loop recommendation policies for networked multi-topic opinion dynamics. The recommendation mechanism was formulated as an infinite-horizon optimal control problem in which recommendation--opinion alignment is rewarded as a proxy for engagement, while polarization, deviation from the uncontrolled equilibrium, recommendation effort, graph-regularization term, and recommendation mismatch are penalized. An agent-wise representation was adopted, so that each agent is described as a local multi-topic subsystem and the social graph appears as a diffusive interconnection among the subsystems.

Within the strictly positive-definite regime, the affine optimal-control problem admits a clear two-layer interpretation. The optimal steady-state pair $(z^\star,s^\star)$ is first obtained from a strictly convex static quadratic problem. After shifting the state and input around this pair, the remaining transient problem reduces to a standard strict LQR. This yields the centralized Riccati controller, which provides the nominal performance benchmark.

Since the Riccati gain is generally dense, we also considered structured affine feedback laws satisfying a prescribed information pattern. In the general formulation, both the static gain $K_s$ and the affine offset $k_s$ are design variables, and each pair induces its own closed-loop equilibrium. In the proposed sequential design, we first synthesize $K_s$ through surrogate optimal control formulations expressed through dissipativity-based or $\mathcal H_2$-based LMIs and then compute the optimal steady-state pair $(z^\star,s^\star)$ used by the centralized controller. The offset is subsequently recovered as $k_s^\star=d_z+(\widetilde A-K_s)z^\star$. Selecting a common steady state provides a natural basis for comparing the controllers, but it remains a design choice that may be adapted to the information available for the local decomposition or to other implementation requirements.

To improve scalability, the centralized structured LMIs were further represented through weighted local contributions defined over (possibly overlapping) certificate clusters. The lifted sum of the local matrices reconstructs the corresponding centralized LMI exactly, while negativity of every local contribution provides a sufficient condition for negativity of the centralized matrix. The choice of the certificate cluster family and of \green{the certificate weights} may therefore affect feasibility and conservatism.

A separate decomposition was introduced for the steady-state problem. In this case, the local quadratic objectives reconstruct the centralized static objective exactly, provided that the diagonal, linear, and off-diagonal contributions satisfy the partition-of-unity conditions and consensus is enforced on shared variables. The steady-state weights do not change the resulting optimizer, although they may influence numerical conditioning and the convergence of a distributed implementation. The steady-state sparsity graph may also be denser than the base-block interaction graph because the term $\widetilde A^\top R\widetilde A$ can generate couplings between blocks that influence a common dynamic equation. The numerical example illustrates the resulting design procedure on a social network with community structure. 

Several limitations remain \green{and motivate further research in this area. The first limitation is that} the present analysis assumes full-state information and a known, time-invariant model specified by the social Laplacian, the topic-coupling matrix, and the anchoring parameters. In practical recommendation systems, these quantities must be inferred from data and may evolve over time based also on external, possibly unpredictable, events. Extensions to partial observations, state estimation, output feedback, uncertain or time-varying networks, and data-driven model identification thus constitute relevant directions for future work.

\green{The second limitation is that} the formulation uses continuous time linear dynamics, quadratic performance indices, static state feedback, and unconstrained recommendation inputs. Incorporating input saturation, finite recommendation sets, quantization, constraints on the support of the opinions, explicit diversity objectives, and nonlinear or stochastic user responses would lead to more realistic models, but would also require constrained, robust, hybrid, or nonlinear control techniques. Further open problems are the systematic selection of the base block partition, the certificate cluster cover, and the associated weights so as to balance computational complexity and certificate conservatism.

\bibliography{Bib_fin}

\end{document}